\documentclass[a4paper, 10pt]{amsart}
\usepackage{amscd}
\usepackage{amssymb}
\usepackage[all]{xy}
\usepackage[T1]{fontenc}
\usepackage{lmodern}
\date{27 Sep 2012}
\title[Deformations of Affine Varieties]
{Deformations of Affine Varieties and the Deligne Crossed Groupoid}
\usepackage{hyperref}
\hypersetup{colorlinks=false}
\author{Amnon Yekutieli}
\address{A. Yekutieli: Department of  Mathematics
Ben Gurion University,
Be'er Sheva 84105,
Israel}
\email{amyekut@math.bgu.ac.il}
\thanks{
Primary: 53D55; Secondary: 14B10, 16S80, 17B40, 18D05.}
\thanks{This research was supported by the US-Israel Binational
Science Foundation and by the Israel Science Foundation.}

\newtheorem{thm}[equation]{Theorem}
\newtheorem{cor}[equation]{Corollary}
\newtheorem{prop}[equation]{Proposition}
\newtheorem{lem}[equation]{Lemma}
\theoremstyle{definition}
\newtheorem{dfn}[equation]{Definition}
\newtheorem{rem}[equation]{Remark}
\newtheorem{exa}[equation]{Example}

\newtheorem{setup}[equation]{Setup}

\numberwithin{equation}{section}
\newcommand{\iso}{\xrightarrow{\simeq}}

\newcommand{\opn}{\operatorname}
\newcommand{\cat}[1]{\operatorname{\mathsf{#1}}}

\newcommand{\ol}{\overline}

\newcommand{\rmitem}[1]{\item[\text{\textup{(#1)}}]}
\newcommand{\mfrak}[1]{\mathfrak{#1}}
\newcommand{\mcal}[1]{\mathcal{#1}}

\newcommand{\mrm}[1]{\mathrm{#1}}
\newcommand{\mbb}[1]{\mathbb{#1}}

\newcommand{\smfrac}[2]{\textstyle \frac{#1}{#2}}
\newcommand{\tup}[1]{\textup{#1}}
\newcommand{\bsym}[1]{\boldsymbol{#1}}
\newcommand{\boplus}{\bigoplus\nolimits}

\newcommand{\bosum}{\sum\nolimits}

\newcommand{\ot}{\otimes}
\newcommand{\hatotimes}[1]{\, \what{{\otimes}}_{#1} \,}
\newcommand{\hot}{\hatotimes{}}
\newcommand{\til}[1]{\tilde{#1}}
\newcommand{\what}[1]{\widehat{#1}}

\newcommand{\K}{\mbb{K}}

\newcommand{\N}{\mbb{N}}
\newcommand{\Z}{\mbb{Z}}
\newcommand{\OO}{\mcal{O}}
\newcommand{\OX}{\mcal{O}_X}
\renewcommand{\AA}{\mcal{A}}

\newcommand{\g}{\mfrak{g}}
\newcommand{\m}{\mfrak{m}}

\newcommand{\om}{\omega}
\newcommand{\al}{\alpha}
\newcommand{\be}{\beta}
\newcommand{\ga}{\gamma}

\newcommand{\bwedge}{{\textstyle \bigwedge}}

\renewcommand{\d}{\mrm{d}}

\newcommand{\End}{\mcal{E}nd}
\newcommand{\Aut}{\mcal{A}ut}
\newcommand{\lb}{\linebreak}
\newcommand{\crvar}{\curvearrowright}

\begin{document}

\begin{abstract}
Let $X$ be a smooth affine algebraic variety over a field $\K$ of characteristic
$0$, and let $R$ be a complete parameter $\K$-algebra (e.g.\ $R = \K[[\hbar]]$).
We consider associative (resp.\ Poisson) $R$-deformations of
the structure sheaf $\OO_X$. The set of $R$-deformations has a crossed
groupoid (i.e.\ strict $2$-groupoid) structure. 
Our main result is that there is a canonical equivalence of crossed
groupoids from the Deligne crossed groupoid of normalized
polydifferential operators (resp.\ polyderivations) of $X$ to the crossed
groupoid of associative (resp.\ Poisson) $R$-deformations of $\OX$. 
The proof relies on a careful study of adically complete sheaves. In the
associative case we also have to use ring theory (Ore localizations) and the
properties of the Hochschild cochain complex. 

The results of this paper extend previous work by various authors. They are
needed for our work on twisted deformation
quantization of algebraic varieties. 
\end{abstract}

\maketitle
\tableofcontents

\setcounter{section}{-1}
\section{Introduction}
\label{sec:Int}

A {\em crossed groupoid} (or strict $2$-groupoid) 
\[ \cat{P} = 
( \cat{P}_1, \cat{P}_2, \opn{Ad}_{\cat{P}_1 \crvar \cat{P}_{2}}, \opn{D} ) \]
consists of groupoids $\cat{P}_1$ and $\cat{P}_2$, such that 
$\opn{Ob}(\cat{P}_1) = \opn{Ob}(\cat{P}_2)$, and $\cat{P}_2$ is totally
disconnected; an action $\opn{Ad}_{\cat{P}_1 \crvar \cat{P}_{2}}$ of 
$\cat{P}_1$ on $\cat{P}_2$ called the twisting; and a morphism of groupoids
(i.e.\ a functor) $\opn{D} : \cat{P}_2  \to \cat{P}_1$ called the feedback.
There are certain conditions -- see Definition \ref{dfn:cosim.101} for full
details. If $\cat{P}$ has only one object, then it is a {\em crossed module}. 
The morphisms in the groupoid $\cat{P}_i$ are called $i$-morphisms. 

Suppose $\cat{P}'$ is another crossed groupoid. A morphism of crossed groupoids
$\Phi : \cat{P} \to \cat{P}'$ is a pair of groupoid morphisms 
$\Phi_i : \cat{P}_i \to \cat{P}'_i$, $i = 1, 2$, that are equal on objects, and
respect the twistings and the feedbacks. We say that $\Phi$ is an
{\em equivalence} if $\Phi_1$ is an equivalence (in the usual sense:
essentially surjective on objects and fully faithful), and $\Phi_2$ is
fully faithful.

Let $\K$ be a field of characteristic $0$. 
A {\em parameter $\K$-algebra} is a complete local noetherian commutative
$\K$-algebra $R$, with maximal ideal $\m$ and residue field $R / \m = \K$.
The important example is $R = \K[[\hbar]]$, the ring of formal power series in
a variable $\hbar$. For $i \in \N$ we let $R_i := R / \m^{i+1}$. So $R_0 = \K$. 

Let $X$ be a smooth algebraic variety over $\K$. An {\em associative
$R$-deformation of $\OX$} is a sheaf $\mcal{A}$ of flat $\m$-adically complete
associative unital $R$-algebras on $X$, with an isomorphism 
$\K \ot_R \mcal{A} \to \OX$ of $\K$-algebras, called an augmentation. 
A  {\em Poisson $R$-deformation of $\OX$} is a sheaf $\mcal{A}$ of flat
$\m$-adically complete Poisson commutative $R$-algebras on $X$, with an
augmentation $\K \ot_R \mcal{A} \to \OX$. A {\em gauge transformation} 
$g : \mcal{A} \to \mcal{A}'$ between $R$-deformations (of the same kind) 
is an isomorphism of $R$-algebras (associative or Poisson) that commutes with
the augmentations. Similarly, for a commutative $\K$-algebra $C$ we consider
associative and Poisson $R$-deformations of $C$.

Let $\mcal{A}$ be an $R$-deformation of $\OX$. The sheaf $\m \mcal{A}$ 
is an $\m$-adically complete sheaf of pronilpotent Lie $R$-algebras (cf.\
Theorem \ref{thm:220}(3) below). The Lie bracket is either the associative
commutator, or the Poisson bracket, as the case may be. There is an associated
sheaf of pronilpotent groups 
$\opn{IG}(\mcal{A}) := \exp(\m \mcal{A})$. 

We denote by $\cat{AssDef}(R, \mcal{O}_X)$ the set of all associative
$R$-deformations of $\OX$, and by $\cat{PoisDef}(R, \mcal{O}_X)$ the set of all
Poisson $R$-deformations of $\OX$. These sets have crossed groupoid
structures on them, where the $1$-morphisms are the gauge transformations 
$g : \mcal{A} \to \mcal{A}'$, and the $2$-morphisms are the elements of the
groups $\Gamma(X, \opn{IG}(\mcal{A}))$. See Proposition \ref{prop:238}.
For an open set $U \subset X$ and a homomorphism $R \to R'$ of parameter
algebras there is a morphism of crossed groupoids 
\[ \cat{AssDef}(R, \mcal{O}_X) \to \cat{AssDef}(R', \mcal{O}_U) \ , \
\mcal{A} \mapsto (R' \hatotimes{R} \mcal{A})|_U \]
and likewise for Poisson deformations. 

Let $\g$ be a {\em quantum type} DG Lie $\K$-algebra, i.e.\ 
$\g = \bigoplus_{i \geq -1} \g^i$. There is an induced Lie $R$-algebra 
$\m \hot \g := \bigoplus_{i \geq -1} \m \hot \g^i$. We denote by 
$\opn{MC}(\m \hot \g)$ the set of solutions of the Maurer-Cartan equation in 
$\m \hot \g$. The {\em Deligne crossed groupoid} $\opn{Del}(\g, R)$
has set of objects $\opn{MC}(\m \hot \g)$, its $1$-morphisms are the elements
of the gauge group $\exp(\m \hot \g^0)$, and its $2$-morphisms are the
elements of the groups $\exp(\m \hot \g^{-1})_{\om}$ for 
$\om \in \opn{MC}(\m \hot \g)$. For full details see Definition
\ref{dfn:Lie-desc.101}.

On the variety $X$ there are sheaves of quantum type DG Lie algebras 
$\mcal{T}_{\mrm{poly}, X}$ and 
$\mcal{D}_{\mrm{poly}, X}^{\mrm{nor}}$, called the sheaves of {\em
polyderivations} and {\em normalized polydifferential operators} respectively.  
Given an affine open set $U \subset X$ we obtain 
quantum type DG Lie algebras 
$\Gamma(U, \mcal{T}_{\mrm{poly}, X})$ and 
$\Gamma(U, \mcal{D}_{\mrm{poly}, X}^{\mrm{nor}})$, 
to which we can apply the Deligne crossed groupoid construction. 

The purpose of this paper is to prove: 

\begin{thm} \label{thm:1}
Let $\K$ be a field of characteristic $0$, $X$ a smooth algebraic variety
over $\K$, $R$ a parameter algebra over $\K$, 
and $U$ an affine open set in $X$.
There are equivalences of crossed groupoids
\[  \opn{geo} : \
\opn{Del} \bigl( \Gamma(U, \mcal{D}_{\mrm{poly}, X}^{\mrm{nor}}) , R \big)
\to \cat{AssDef}(R, \mcal{O}_U) \]
and
\[  \opn{geo} : \
\opn{Del} \bigl( \Gamma(U, \mcal{T}_{\mrm{poly}, X}) , R \big)
\to \cat{PoisDef}(R, \mcal{O}_U) \]
which we call {\em geometrization}. The equivalences $\opn{geo}$ commute with 
homomorphisms $R \to R'$ of parameter algebras, and with inclusions of affine
open sets $U' \to U$.
\end{thm}

This theorem is a key ingredient in our proof of {\em twisted deformation
quantization} in \cite{Ye6}. 
(Actually this paper was once a part of \cite{Ye6}, but we now decided to make
it into a separate paper, since \cite{Ye6} was becoming too long.)

Theorem \ref{thm:1} is repeated as Theorem \ref{thm:202} in the Section
\ref{sec:defs-vars} of the paper, and is proved there. The proof requires
several intermediate results, and we now list some of them.  

The first intermediate result does not require $X$ to be an algebraic variety,
nor for $\K$ to have characteristic $0$. 
What we need is that $X$ is a topological space; $\mcal{M}_0$ is a sheaf of
$\K$-modules on $X$; and $X$ has {\em enough $\mcal{M}_0$-acyclic open
sets}. This is explained in Definition \ref{dfn:23}. An algebraic variety
has enough $\OX$-acyclic open sets: the affine open sets. 

\begin{thm} \label{thm:220}
Let $X$ be a topological space, $U \subset X$ an open set, $\K$ a field, 
$(R, \m)$ a parameter $\K$-algebra, and $\mcal{M}$ a sheaf of $R$-modules on
$X$. Define $M := \Gamma(U, \mcal{M})$
and $\mcal{M}_0 := \K \ot_R \mcal{M}$. 
Assume that $\mcal{M}$ is flat over $R$ and $\m$-adically complete, $X$ has
enough $\mcal{M}_0$-acyclic open sets, and $U$ is $\mcal{M}_0$-acyclic. 
Then\tup{:}
\begin{enumerate}
\item The $R$-module $M$ is flat and $\m$-adically complete.

\item Take any $i \in \N$. 
\begin{enumerate}
\item The canonical homomorphism
$R_i \ot_R M \to  \Gamma(U, R_i \ot_R \mcal{M})$ is bijective.

\item The $R$-module $\m^i M$ is $\m$-adically complete.

\item The sheaf of $R$-modules $\m^i \mcal{M}$ is $\m$-adically complete.

\item The canonical homomorphism
$\m^i M \to \Gamma(U, \m^i \mcal{M})$ 
is bijective.
\end{enumerate}
\end{enumerate}
\end{thm}

This is a combination of Theorem \ref{thm:230} and
Corollaries \ref{cor:241} and \ref{cor:242} in the body of the paper. The proofs
use results from \cite{Ye4}. 

The second intermediate result also does not require characteristic $0$. It is a
statement about sheaves of noncommutative rings on algebraic varieties. 
Suppose $\mcal{A}$ is an $R$-deformation of $\OX$, and $U \subset X$ is an
affine open set. Let $A := \Gamma(U, \mcal{A})$ and 
$C := \Gamma(U, \mcal{O}_X)$. According to Theorem \ref{thm:220}, 
the $R$-module $A$ is flat and complete, and $\K \ot_R A \cong C$; so $A$ is an
$R$-deformation of $C$. 

\begin{thm} \label{thm:221}
Let $\K$ be a field, $X$ an algebraic variety over $\K$, $U$ an affine open
set of $X$, and $C := \Gamma(X, \mcal{O}_X)$. 
\begin{enumerate}
\item Let $A$ be an associative $R$-deformation of $C$. Then there exists an 
associative $R$-deformation $\mcal{A}$ of $\mcal{O}_U$, together with a 
gauge transformation of deformations
$g : A \to \Gamma(U, \mcal{A})$. 

\item Let $\mcal{A}$ and $\mcal{A}'$ be associative $R$-deformations of
$\mcal{O}_U$, and let $g : \Gamma(U, \mcal{A}) \to \Gamma(U, \mcal{A}')$
gauge transformation of deformations of $C$. Then there is a unique 
gauge transformation of deformations
$\til{g} : \mcal{A} \to \mcal{A}'$
such that $\Gamma(U, \til{g}) = g$.
\end{enumerate}
\end{thm}

Observe that part (2) implies that the deformation $\mcal{A}$ in part (1) is
unique up to a unique isomorphism. 

This result is repeated as Theorem \ref{thm:defs-sh.112} in the body of the
paper, and proved there. The proof relies on a detailed study of the Ore
localizations that are related to associative deformations. 

For the third intermediate result we again assume that the base field $\K$ has
characteristic $0$. Let $C$ be a smooth commutative $\K$ algebra (namely 
$U := \opn{Spec} C$ is a smooth algebraic variety over $\K$). 
We consider $A := R \hot C$ as an $R$-module, equipped with a distinguished
element $1_A := 1_R \ot 1_C$ and an augmentation $A \to C$. 
A {\em star product} on $A$ is an $R$-bilinear unital associative multiplication
$\star$, with unit $1_A$, that lifts the original multiplication
$c_1 \cdot c_2$ of $C$. 
Thus $(A, \star)$ is an associative $R$-deformation of $C$. 

By {\em gauge transformation} of the augmented $R$-module $A$ we mean an
isomorphism of $R$-modules $g : A \to A$, that commutes with the augmentation to
$C$ and fixes $1_A$. 
Suppose $\star$ and $\star'$ are two star products on $A$. We say that 
$g$ is a gauge transformation from $\star$ to $\star'$ if 
\begin{equation} \label{eqn:225}
g(a_1 \star a_2) = g(a_1) \star' g(a_2)  
\end{equation}
for all $a_1, a_2 \in A$.

A star product $\star$ is called {\em differential} if there is a (unique)
element 
$\om \in \lb \opn{MC}(\m \hot \mcal{D}_{\mrm{poly}}^{\mrm{nor}}(C))$
such that 
$c_1 \star c_2 = c_1 \cdot c_2 + \om(c_1, c_2)$
for all $c_1, c_2 \in C$. 
A gauge transformation $g :A \to A$ is called differential 
if $g = \exp(\ga)$ for some (unique)
$\ga \in  \m \hot \mcal{D}_{\mrm{poly}}^{\mrm{nor}, 0}(C)$.
Here 
$\mcal{D}_{\mrm{poly}}^{\mrm{nor}}(C) := 
\Gamma(U, \mcal{D}_{\mrm{poly}, U}^{\mrm{nor}})$.

\begin{thm} \label{thm:225}
Let $\K$ be a field of characteristic $0$, $C$ a smooth $\K$-algebra, and 
$R$ a parameter $\K$-algebra. 
Consider the augmented $R$-module $A := R \hot C$ with distinguished element 
$1_A$. 
\begin{enumerate}
\item Any star product on $A$ is gauge equivalent to a differential star
product. Namely, given a star product $\star$ on $A$, there exists a 
differential star product $\star'$, and a gauge transformation $g : A \to A$,
such that equation \tup{(\ref{eqn:225})} holds.

\item Let $\star$ and $\star'$ be star products on $A$, and let $g$ be a gauge
transformation of $A$ satisfying \tup{(\ref{eqn:225})}. Assume that $\star$ is a
differential star product. The following conditions are equivalent\tup{:}
\begin{enumerate}
\rmitem{i} The star product $\star'$ is also differential.

\rmitem{ii} The gauge transformation  $g$ is differential.
\end{enumerate}
\end{enumerate}
\end{thm}

This is a combination of Theorems \ref{thm:4} and \ref{thm:5} in Section 
\ref{sec:ploy-diff}. 
It relies on results from \cite{Ye2} on the structure of the DG Lie algebra 
$\mcal{D}_{\mrm{poly}}^{\mrm{nor}}(C)$.
Part (2) was communicated to us by P. Etingof; it is similar to 
\cite[Proposition 2.2.3]{KS}.

Let us now discuss how this paper relates to other work in this field. The role
of the DG Lie algebras  $\mcal{T}_{\mrm{poly}}$ and 
$\mcal{D}_{\mrm{poly}}$ in deformation quantization goes back a long time; 
most notably it figured in the groundbreaking paper \cite{Ko1} of M. Kontsevich
from 1997. See also the papers \cite{Ko1, CKTB, BGNT, Ye1, VdB} and the
references therein. 

Complete deformations (i.e.\ $R$-deformations where $R$ is a complete ring)
were not treated properly before, with the exception of the work of M. Kashiwara
and P. Schapira, who considered $R = \K[[\hbar]]$ (see \cite{KS} and other
papers). Most authors just dealt with nilpotent deformations (i.e.\ $R$ is 
an artinian ring). Our own work in \cite{Ye1} was flawed in this respect --
see \cite[Remark 8.14]{Ye5}. Indeed the papers \cite{Ye4} and \cite{Ye5} came
into existence to remedy this flaw! The present paper and \cite{Ye6} attempt
to provide a correct treatment of complete deformations and their twisted
versions. 

The differential aspect of associative deformations of smooth affine varieties
(Theorem \ref{thm:225}) was 
not well-understood previously. In our paper \cite{Ye1} we demanded as a
condition that associative deformations should be locally differential
(cf.\ \cite[Definition 1.6]{Ye1}). 
Due to Theorems \ref{thm:220} and \ref{thm:225} we now know that this condition
is redundant. It is interesting to note that in the complex analytic case the
question is still open (see \cite[Remark 2.2.7]{KS}). 
 
Crossed groupoids (or $2$-groupoids) appeared in this subject already
in 1994 -- see P. Deligne's letter to L. Breen \cite{De}, and Breen's
classification
of gerbes in terms of crossed modules \cite{Br}. A more recent use of crossed
groupoids to classify stacks on a topological space can be found in \cite{DP}. 
The papers \cite{De, Ge, BGNT} only treated the Deligne crossed groupoid of a
DG Lie algebra.
As far as we know there is nothing in prior literature resembling Theorem
\ref{thm:1}, namely giving the equivalence from the Deligne
crossed groupoid to the crossed groupoid of geometric origin 
$\cat{AssDef}(R, \mcal{O}_X)$ (resp.\ $\cat{PoisDef}(R, \mcal{O}_X)$) -- even 
for nilpotent parameters.
As already mentioned, this equivalence is crucial
for proving twisted deformation quantization in \cite{Ye6}.

\medskip \noindent
\textbf{Acknowledgments.}
Work on this paper began together with Fredrick Leitner, and I wish to
thank him for his contributions. Many of the ideas in this paper are influenced
by the work of Maxim Kontsevich, and I am grateful to him for discussing this
material with me. Thanks also to Michael Artin, Pavel Etingof, Damien
Calaque, Michel Van den Bergh, Pierre Deligne, Lawrence Breen, Pierre Schapira,
James Stasheff, Pietro Polesello and Matan Prezma
for their assistance on various aspects of the paper.

%\cleardoublepage
\section{Crossed Groupoids}
\label{sec:crossed}
\numberwithin{equation}{section}

In this section we review the categorical (or combinatorial) concept of
{\em crossed groupoid}.  

Let $G$ be a groupoid (i.e.\ a category in which all morphisms are invertible),
with set of objects $\opn{Ob}(G)$. Given $\om, \om' \in \opn{Ob}(G)$
we denote by $G(\om, \om) := \opn{Hom}_G(\om, \om')$, the set of morphisms. 
We also write $G(\om) := G(\om, \om)$, the automorphism group of the object
$\om$. For $g \in G(\om, \om')$ and $h \in G(\om)$ we let 
\begin{equation} \label{equ:240}
\opn{Ad}_G(g)(h) := g \circ h \circ g^{-1} \in G(\om') .
\end{equation}

Suppose $N$ is another groupoid, such that 
$\opn{Ob}(N) = \opn{Ob}(G)$. 
An {\em action} $\Psi$ of $G$ on 
$N$ is a collection of group isomorphisms 
$\Psi(g) : N(\om) \iso N(\om')$
for all $\om, \om' \in \opn{Ob}(G)$ and $g \in G(\om, \om')$, 
such that 
$\Psi(h \circ g) = \Psi(h) \circ \Psi(g)$
whenever $g$ and $h$ are composable, and 
$\Psi(1_{\om})$ is the identity automorphism of $N(\om)$.
For instance, there is the action $\opn{Ad}_G$ of $G$ on itself, described in
equation (\ref{equ:240}).

\begin{dfn} \label{dfn:cosim.101}
A {\em crossed groupoid} is a structure 
\[ G = 
( G_1, G_2, 
\opn{Ad}_{G_1 \crvar G_{2}}, \opn{D} ) \]
consisting of:
\begin{itemize}
\item Groupoids $G_1$ and $G_2$, such that 
$G_2$ is totally disconnected, and 
$\opn{Ob}(G_1) = \opn{Ob}(G_2)$. 
We write $\opn{Ob}(G) := \opn{Ob}(G_1)$. 

\item An action $\opn{Ad}_{G_1 \crvar G_{2}}$ of 
$G_1$ on $G_2$, called the {\em twisting}. 

\item A morphism of groupoids (i.e.\ a functor) 
$\opn{D} : G_2 \to G_1$
called the {\em feedback}, which is the identity on objects.
\end{itemize}

These are the conditions:
\begin{enumerate}
\rmitem{i} The morphism $\opn{D}$ is $G_1$-equivariant with respect to
the actions $\opn{Ad}_{G_1 \crvar G_{2}}$ and
$\opn{Ad}_{G_1}$. Namely 
\[ \opn{D}(\opn{Ad}_{G_1 \crvar G_{2}}(g)(a)) =
\opn{Ad}_{G_1}(g)(\opn{D}(a)) \]
in the group $G_1(\om')$, for any $\om, \om' \in \opn{Ob}(G)$, 
$g \in G_1(\om, \om')$ and $a \in G_2(\om)$.

\rmitem{ii} For any $\om \in \opn{Ob}(G)$ and 
$a \in G_2(\om)$ there is equality
\[ \opn{Ad}_{G_1 \crvar G_{2}}(\opn{D}(a)) =
\opn{Ad}_{G_2(\om)}(a) , \]
as automorphisms of the group $G_2(\om)$.
\end{enumerate}
\end{dfn}

We sometimes refer to morphisms in the groupoid $G_1$ as {\em
$1$-morphisms}, or as {\em gauge transformations}. 
For an object $\om \in \opn{Ob}(G)$, elements of the group $G_2(\om)$ 
are sometimes called {\em $2$-morphisms} or {\em inner gauge transformations}.
The groupoid $G_1$ is called the {\em $1$-truncation} of the crossed groupoid
$G$.

\begin{exa}
Consider any groupoid $G_1$, and let $G_2$ be the associated totally
disconnected
groupoid (gotten be removing all morphisms between distinct objects). 
(More generally one can take a normal subgroupoid $N \subset G$, in the sense 
of \cite[Definition 3.1]{Ye3}, and define $G_2 := N$.)
Define the twisting  $\opn{Ad}_{G_1 \crvar G_{2}} := \opn{Ad}_{G_1}$,
and the feedback $\opn{D}$ is the inclusion. 
This is easily seen to be a crossed groupoid. 
\end{exa}

Let $\cat{Grp}$ be the category of groups. For a groupoid $G$ there is a
functor 
\begin{equation} \label{eqn:300}
\opn{Aut}_G : G \to \cat{Grp}
\end{equation}
which on objects is $\opn{Aut}_G(\om) := G(\om)$. 
For a morphism $g : \om \to \om'$ in $G$ the group isomorphism 
$\opn{Aut}_G(g) : \opn{Aut}_G(\om) \to \opn{Aut}_G(\om')$ is 
$\opn{Aut}_G(g) := \opn{Ad}_G(g)$, cf.\ (\ref{equ:240}). 

\begin{prop} \label{prop:200}
Let 
$G = ( G_1, G_2, \opn{Ad}_{G_1 \crvar G_{2}}, \opn{D} )$
be a crossed groupoid. 
\begin{enumerate}
\item For $i \in \opn{Ob}(G)$ let $\opn{IG}(i) := G_2(i)$, and for 
$g \in G_1(i, j)$ let 
$\opn{IG}(g) := \opn{Ad}_{G_1 \crvar G_{2}}(g)$. 
Then 
\[ \opn{IG} : G_1 \to \cat{Grp} \]
is a functor. 

\item For $i \in \opn{Ob}(G)$ and $a \in \opn{IG}(i)$ let 
$\opn{ig}(a) := \opn{D}(a) \in G_1(i)$. Then 
\[ \opn{ig} : \opn{IG} \to \opn{Aut}_{G_1} \]
is a natural transformation of functors $G_1 \to \cat{Grp}$. 

\item The data $( G_1, G_2, \opn{Ad}_{G_1 \crvar G_{2}}, \opn{D} )$
can be recovered from the groupoid $G_1$, the functor 
$\opn{IG} : G_1 \to \cat{Grp}$ and the 
natural transformation 
$\opn{ig} : \opn{IG} \to \opn{Aut}_{G_1}$.
\end{enumerate}
\end{prop}

\begin{proof}
This is immediate from the definitions. 
\end{proof}

\begin{dfn} \label{dfn:220}
Suppose 
$H = ( H_1, H_2, 
\opn{Ad}_{H_1 \crvar H_{2}}, \opn{D} )$ 
is another crossed groupoid. A {\em morphism of crossed groupoids}
$\Phi : G \to H$
is a pair of morphisms of groupoids 
$\Phi_i : G_i \to H_i$, $i = 1, 2$, that are equal on objects, 
and respect the twistings and the feedbacks. 
\end{dfn}

\begin{dfn} \label{dfn:221}
A morphism of crossed groupoids $\Phi : G \to H$ is called an
{\em equivalence} if it satisfies these conditions:
\begin{enumerate}
\rmitem{i} $\Phi_1 : G_1 \to H_1$ is an equivalence of groupoids; namely 
it is essentially surjective on objects, and for every 
$\om \in \opn{Ob}(G)$ the group homomorphism 
$\Phi_1 : G_1(\om) \to H_1(\Phi(\om))$
is bijective. 
 
\rmitem{ii} For any $\om \in \opn{Ob}(G)$ the group homomorphism 
$\Phi_2 : G_2(\om) \to H_2(\Phi(\om))$
is bijective. 
\end{enumerate}
\end{dfn}

It is easy to see that an equivalence of crossed groupoids
$\Phi : G \to H$ admits a quasi-inverse $H \to G$ (we leave it to the reader to
spell out what this means).

\begin{rem} \label{rem:cosim.101}
A crossed groupoid is better known as a {\em strict $2$-groupoid},
or a {\em crossed module over a groupoid}, or a {\em $2$-truncated crossed
complex}; see \cite{Bw, Ye5}. When $\opn{Ob}(G)$ is a singleton then $G$ is just
a crossed module. 

Traditionally papers used $2$-groupoid language to discuss descent (cf.\
\cite{De} and \cite{BGNT}). In our work we realized that the crossed groupoid
language is
more effective and natural in this context.
\end{rem}

\begin{rem}
We ignore issue of set theory (like the size of the set of objects $\opn{Ob}(G)$
of a groupoid $G$). The blanket assumptions we rely on are explained in 
\cite[Section 1]{Ye3}.
\end{rem}

%\cleardoublepage
\section{Deformations of Algebras}
\label{sec:defs-alg}
\numberwithin{equation}{section}

In this section we give the basic definitions and a few initial
results. 

Here, and in the rest of the paper, we work over a base field $\K$. All
algebras are by default $\K$-algebras, and all homomorphism between
algebras are over $\K$. For $\K$-modules $M, N$ we write 
$M \ot N := M \ot_{\K} N$  and
$\opn{Hom}(M, N) := \opn{Hom}_{\K}(M, N)$.
By default, associative algebras are assumed to be unital, and
commutative algebras are assumed to be associative (and unital).
Homomorphisms between unital algebras always preserve units.

\begin{dfn} \label{dfn:13}
A {\em parameter $\K$-algebra} is a 
complete local noetherian commutative
$\K$-algebra $R$, with maximal ideal $\m$ and residue field 
$R / \m = \K$.
We sometimes say that $(R, \m)$ is a parameter $\K$-algebra.
For $i \geq 0$ we let
$R_i := R / \m^{i+1}$.
The $\K$-algebra homomorphism $R \to \K$ is called the {\em
augmentation} of $R$.

Suppose $(R', \m')$ is another parameter $\K$-algebra. 
By {\em homomorphism of parameter algebras} we mean a 
$\K$-algebra homomorphism $\sigma : R \to  R'$.
\end{dfn}

Note that $R$ can be recovered from $\m$, since
$R = \K \oplus \m$ as $\K$-modules, with the obvious multiplication. 
A homomorphism $\sigma : R \to  R'$ necessarily satisfies 
$\sigma(\m) \subset \m'$; so letting $R'_i  := R' / \m'^{\, i+1}$, 
there is an induced homomorphism $R_i \to R'_i$.

\begin{exa}
The most important parameter algebra in deformation theory is
$\K[[\hbar]]$, the ring of formal power series in the variable
$\hbar$. A $\K[[\hbar]]$-deformation (see below) is sometimes called
a ``$1$-parameter formal deformation''. 
\end{exa}

Let $M$ be an $R$-module. For any $i \geq 0$ there is a canonical bijection
$R_i \otimes_R M \cong M / \m^{i+1} M$. 
The $\m$-adic completion of $M$ is the $R$-module
$\what{M} :=  \lim_{\leftarrow i}\, (R_i \otimes_R M)$. 
The module $M$ is called {\em $\m$-adically complete} if the
canonical homomorphism
$M \to \what{M}$ is bijective. (Some texts, including \cite{Bo1}, 
would say that  ``$M$ is separated and complete''.)
Since $R$ is noetherian, the $\m$-adic completion of any $R$-module is 
$\m$-adically complete; see \cite[Corollary 3.5]{Ye4}. (This may be false when
$R$ is not noetherian.) 

Given a $\K$-module $V$ and an $R$-module $M$, we let
$M \hot V := \what{M \ot V}$, the $\m$-adic completion of the $R$-module 
$M \ot V$.  

\begin{dfn} \label{dfn:31}
Let $(R, \m)$ be a parameter $\K$-algebra.
An {\em $\m$-adic system of $R$-modules} is a collection 
$\{ M_i \}_{i \in \mbb{N}}$ of $R$-modules, together with a
collection $\{ \psi_i \}_{i \in \mbb{N}}$ of 
homomorphisms 
$\psi_i : M_{i+1} \to M_i$.
The conditions are:
\begin{enumerate}
\rmitem{i} For every $i$ one has $\m^{i+1} M_i = 0$. Thus 
$M_i$ is an $R_i$-module.
\rmitem{ii} For every $i$ the $R_i$-linear homomorphism
$R_i \otimes_{R_{i+1}} M_{i+1} \to M_i$
induced by $\psi_i$ is an isomorphism. 
\end{enumerate}
\end{dfn}

The following (not so well known) facts will be important for us.

\begin{prop} \label{prop:12}
Let $(R, \m)$ be a parameter $\K$-algebra, and let $M$ be an
$R$-module. Define 
$M_i := R_i \otimes_R M$.
The following conditions are equivalent\tup{:}
\begin{enumerate}
\rmitem{i} The $R$-module $M$ is flat and $\m$-adically complete.

\rmitem{ii} There is an $\m$-adic system of $R$-modules
$\{ N_i  \}_{i \in \mbb{N}}$, such that each $N_i$ is flat over
$R_i$, and an isomorphism of $R$-modules
$M \cong \lim_{\leftarrow i}\, N_i$.

\rmitem{iii} There is an isomorphism of $R$-modules 
$M \cong R \hot V$ for some $\K$-module $V$.

\rmitem{iv} The $R$-module $M$ is $\m$-adically complete, and for
every $\K$-linear homomorphism $M_0 \to M$ splitting the canonical
surjection $M \to M_0$, the induced $R$-linear homomorphism
$R \hot M_0 \to M$ is bijective.
\end{enumerate}

Moreover, when these conditions hold, the induced homomorphisms
$R_i \ot V \to M_i \to N_i$
are bijective for every $i$.
\end{prop}

\begin{proof}
When $M$ is finitely generated or $\m$ is nilpotent, the equivalence of
conditions (i), (ii) and (iii) is \cite[Corollary II.3.2]{Bo1}. For the general
case we need the results of \cite{Ye4}. The module $R \hot V$ is the $\m$-adic
completion of the free $R$-module 
$R \ot V$; so by \cite[Proposition 3.13]{Ye4} the module $R \hot V$ is
$\m$-adically free \cite[Definition 3.11]{Ye4}. Now \cite[Corollary 4.5]{Ye4}
says that conditions (i), (ii) and (iii) are equivalent. 

Since $\K$-linear splittings $M_0 \to M$ exist, condition (iv)
directly implies condition (iii), with $V := M_0$. 
As for the converse, assume that $M$ is $\m$-adically free, and take any
splitting $M_0 \to M$. We get an $R$-linear homomorphism 
$\phi : R \hot M_0 \to M$ lifting $\bsym{1}_{M_0}$.
By the Complete Nakayama \cite[Theorem 2.11]{Ye4}, $\phi$ is surjective.
Since $M$ is $\m$-adically free, there is a homomorphism 
$\psi : M \to R \hot M_0$ that's a right inverse to $\phi$, i.e.\ 
$\phi \circ \psi = \bsym{1}_M$. 
But $\psi$ also lifts $\bsym{1}_{M_0}$, so $\psi$ is surjective. 
We see that $\psi$ is bijective and $\phi = \psi^{-1}$. 
 
The last assertion is a consequence of \cite[Theorem 4.3]{Ye4}.
\end{proof}

Suppose $A$ is an $R$-algebra. We say $A$ is $\m$-adically complete,
or flat, if it is so as an $R$-module. 

\begin{dfn} \label{dfn:1}
Let $\K$ be a field, $(R, \m)$ a parameter $\K$-algebra, and $C$ a commutative
$\K$-algebra. An {\em associative $R$-deformation of $C$} is a flat
$\m$-adically complete associative $R$-algebra $A$, together
with a $\K$-algebra isomorphism
$\psi : \K \otimes_R A \to C$, called an {\em augmentation}.

Given another such deformation $A'$, a {\em gauge transformation}
$g : A \to A'$ is an $R$-algebra isomorphism that commutes with the
augmentations to $C$.

We denote by  $\cat{AssDef}(R, C)$ the groupoid of associative
$R$-deformations of $C$, and gauge transformations between them.
\end{dfn}

Note that an associative $R$-deformation $A$ of $C$ is a unital algebra (by our
conventions), and a gauge transformation $g : A \to A'$ sends the unit 
$1_A$ to the unit $1_{A'}$.

Due to Proposition \ref{prop:12}, for any associative $R$-deformation $A$ of $C$
there exists an isomorphism of augmented $R$-modules 
$R \hot C \iso A$, sending $1_R \otimes 1_C \mapsto 1_A$.

Let $A$ be a commutative $R$-algebra. An $R$-bilinear Poisson bracket
on $A$ is an $R$-bilinear function
$\{ -,- \} : A \times A \to A$
which is a Lie bracket (i.e.\ it is antisymmetric and satisfies
the Jacobi identity), and also
is a derivation in each of its arguments. The pair $(A, \{ -,- \})$
is called a {\em Poisson $R$-algebra}. A homomorphism of Poisson
$R$-algebras $f : A \to A'$ is an algebra homomorphism that respects
the Poisson brackets. 

\begin{dfn} \label{dfn:2}
Let $\K$ be a field, $(R, \m)$ a parameter $\K$-algebra, and $C$ a commutative
$\K$-algebra. We consider $C$ as a Poisson $\K$-algebra with the zero bracket.
A {\em Poisson $R$-deformation of $C$} is a
flat $\m$-adically complete Poisson $R$-algebra $A$,
together with an isomorphism of Poisson $\K$-algebras
$\psi : \K \otimes_R A \to C$, called an {\em augmentation}.

Given another such deformation $A'$, a {\em gauge transformation}
$g : A \to A'$
is an $R$-algebra isomorphism that respects the Poisson brackets and
commutes with the augmentations to $C$.

We denote by  $\cat{PoisDef}(R, C)$ the groupoid of Poisson
$R$-deformations of $C$, and gauge transformations between them.
\end{dfn}

\begin{rem}
If the ring $C$ is noetherian, then any Poisson or associative
$R$-deformation of $C$ is also a (left and right) noetherian ring. See
\cite{KS} or \cite{Bo1}. We are not going to need this fact.
\end{rem}

Suppose $(R', \m')$ is another parameter $\K$-algebra, 
and $\sigma : R \to  R'$ is a $\K$-algebra homomorphism. Given an $R$-module $M$
we let
\[ R' \hatotimes{R} M := 
\lim_{\leftarrow i}\, (R'_i \otimes_R M) . \]
This is the $\m'$-adic completion of the $R'$-module
$R' \otimes_R M$. 

\begin{prop} \label{prop:15}
Let $A$ be an associative \tup{(}resp.\ Poisson\tup{)} 
$R$-deformation of $C$, let 
$R'$ be another parameter $\K$-algebra, let
$\sigma : R \to  R'$ be a $\K$-algebra homomorphism,
and let $A' := R' \hatotimes{R} A$. Then $A'$ has a unique structure
of associative \tup{(}resp.\ Poisson\tup{)} $R'$-deformation of $C$,
such that the canonical homomorphism $A \to A'$ 
is a homomorphism of $R$-algebras \tup{(}resp.\ Poisson
$R$-algebras\tup{)}.
\end{prop}

\begin{proof}
Let  $A'_i := R'_i \otimes_R A$. This is a flat $R'_i$-module,
and it has an induced $R'_i$-bilinear multiplication
(resp.\ Poisson bracket). Thus $A'_i$ is an $R'_i$-deformation of $C$.
In the limit, the $R'$-module
$A' = \lim_{\leftarrow i}\, A'_i$ 
has an induced $R'$-bilinear multiplication
(resp.\ Poisson bracket). 
By Proposition \ref{prop:12} the $R'$-module $A'$ is flat and $\m'$-adically
complete; so $A'$ is an $R'$-deformation of $C$.
\end{proof}

Let $C'$ be another commutative $\K$-algebra, and let 
$\tau : C \to C'$ be a homomorphism. We say that $C'$ is
a {\em principal localization} of $C$ if there is a $C$-algebra
isomorphism $C' \cong C_s = C[s^{-1}]$
for some element $s \in C$.

\begin{thm} \label{thm:12}
Let $\K$ be a field, $R$ a parameter $\K$-algebra, $C$ a commutative
$\K$-algebra, and $A$ a Poisson \tup{(}resp.\
associative\tup{)} $R$-deformation of $C$. Suppose $\tau : C \to C'$
is a principal localization. Then\tup{:}
\begin{enumerate}
\item There exists a 
Poisson \tup{(}resp.\ associative\tup{)} $R$-deformation 
$A'$ of $C'$, together with a homomorphism 
$g : A \to A'$ of 
Poisson \tup{(}resp.\ associative\tup{)} $R$-algebras
which lifts $\tau : C \to C'$.
\item Suppose $\tau' : C' \to C''$ is a homomorphism of commutative
$\K$-algebras, $A''$ is a Poisson \tup{(}resp.\ associative\tup{)}
$R$-deformation of $C''$, and
$h : A \to A''$ is a homomorphism of 
Poisson \tup{(}resp.\ associative\tup{)} $R$-algebras
which lifts $\tau' \circ \tau : C \to C''$.
Then there is a unique  homomorphism of 
Poisson \tup{(}resp.\ associative\tup{)} $R$-algebras
$g' : A' \to A''$ such that 
$h = g' \circ g$.
\end{enumerate}
\end{thm}

When we say that $g : A \to A'$ lifts $\tau : C \to C'$, we mean 
relative to the augmentations $A \to C$ and $A' \to C'$.
Observe that by part (2), the pair $(A' ,g)$ in part (1) is unique up to a
unique gauge transformation.

For the proof we need the next lemma on Ore localization of
noncommutative rings \cite{MR}. Recall that a subset $S$
of a ring $A$ is called a {\em denominator set} if it is
multiplicatively
closed, and satisfies the left and right torsion and Ore conditions. 
If $S$ is a denominator set, then $A$ can be localized with respect
to $S$. Namely there is a ring $A_S$, called the {\em ring of
fractions}, with a ring homomorphism $A \to A_S$. The elements of
$S$ become invertible in $A_S$, and $A_S$ is universal for this
property; every element $b \in A_S$ can be
written as $b = a_1 s_1^{-1} = s_2^{-1} a_2$, with 
$a_1, a_2 \in A$ and $s_1, s_2 \in S$;
and $A_S$ is flat over $A$ (on both sides).

\begin{lem} \label{lem:1}
Let $A$ be a ring, with nilpotent two-sided ideal $\mfrak{a}$.
Assume the ring
$\opn{gr}_{\mfrak{a}} (A) = \boplus_{i \geq 0} 
\mfrak{a}^i / \mfrak{a}^{i+1}$
is commutative. Let $s$ be some element of $A$.
\begin{enumerate}
\item The set $\{ s^j \}_{j \geq 0}$ is a denominator set in $A$.
We denote by $A_s$ the resulting ring of fractions. 
\item Let $\bar{A} := A / \mfrak{a} = \opn{gr}_{\mfrak{a}}^0 (A)$,
let $\bar{s}$ be the image of $s$ in $\bar{A}$, and let
$\mfrak{a}_s$ be the kernel of the canonical ring surjection
$A_s \to \bar{A}_{\bar{s}}$. 
Then $\mfrak{a}_s = \mfrak{a} A_s = A_s \mfrak{a}$,
and this is a nilpotent ideal.
\item Let $a$ be any element of $A$, with image 
$\bar{a} \in \bar{A}$. Then $a$ is invertible in 
$A_s$ if and only if $\bar{a}$ is invertible in 
$\bar{A}_{\bar{s}}$. 
\end{enumerate}
\end{lem}

\begin{proof}
(1) This is a variant of \cite[Corollary 5.18]{YZ}.
We view $A$ as a bimodule over the ring $\mbb{Z}[s]$. Since the 
$\mfrak{a}$-adic filtration is finite, and 
$\opn{gr}_{\mfrak{a}} (A)$ is commutative, it follows from 
\cite[Lemma 5.9]{YZ} that $A$ is evenly localizable to $\mbb{Z}[s, s^{-1}]$.
According to \cite[Theorem 5.11]{YZ} the set 
$\{ s^j \}_{j \geq 0}$ is a denominator set in $A$. Moreover, 
$A_s \cong A \otimes_{\Z[s]} \Z[s, s^{-1}]$
as left $A$-modules.

\medskip \noindent
(2) Since $A \to A_s$ is flat it follows that 
$\mfrak{a}_s = \mfrak{a} A_s = A_s \mfrak{a}$.
By induction on $i$ one then shows that 
$(\mfrak{a}_s)^i = \mfrak{a}^i A_s$; and hence $\mfrak{a}_s$ is
nilpotent.

\smallskip \noindent
(3) We prove only the nontrivial part. Suppose $\bar{a}$ is invertible
in
$\bar{A}_{\bar{s}}$. So 
$\bar{a} \bar{b} = 1$ for some $b \in A_s$. 
Thus
$a b = 1 - \epsilon$ in $A_s$, where $\epsilon \in \mfrak{a}_s$.
Since the ideal $\mfrak{a}_s$ is nilpotent, the element 
$1 - \epsilon$ is invertible in $A_s$.
This proves that $a$ has a right inverse. Similarly for a left
inverse.
\end{proof}

\begin{proof}[Proof of Theorem \tup{\ref{thm:12}}]
The proof is in several steps.

\smallskip \noindent
Step 1. Consider the associative case, and assume $R$ is artinian
(i.e.\ $\m$ is nilpotent). Take an element
$s \in C$ such that $C' \cong C_s$. 
Choose some lifting $\til{s} \in A$ of $s$.
According to Lemma \ref{lem:1} there is a ring of fractions
$A_{\til{s}}$ of $A$, gotten by inverting $\til{s}$ on one side, and 
$\K \otimes_R A_{\til{s}} \cong  C'$.
Since $R$ is central in $A$, it is also central in $A_{\til{s}}$. 
And since $A_{\til{s}}$ is flat over $A$, it is also flat over $R$. We
see that $A_{\til{s}}$ is an associative $R$-deformation of
$C'$, and the homomorphism $g : A \to A_{\til{s}}$ lifts
$C \to C_s$. 

Now suppose we are in the situation of part (2). 
Since 
$(\tau' \circ \tau)(s)$ is invertible in $C''$, Lemma
\ref{lem:1}(3) says that the element
$h(\til{s})$ is invertible in $A''$.
Therefore there is a unique $A$-ring homomorphism
$g' : A_{\til{s}} \to A''$
such that $h = g' \circ g$. 

\smallskip \noindent
Step 2. $R$ is still artinian, but now we are in the Poisson case. 
So $A$ is a Poisson $R$-deformation of $C$. From the previous step we
obtain a flat commutative $R$-algebra $A'$, such that 
$\K \otimes_R A' \cong C'$, together with a homomorphism
$g : A \to A'$. The pair $(A', g)$ is unique for this property. We
have to address the Poisson bracket. 

Take an element $\til{s} \in A$ like in Step 1; so
$A' \cong A_{\til{s}}$. There is a unique
biderivation on the commutative ring $A'$ that
extends the given Poisson bracket 
$\{ -,- \}$ on $A$; it has the usual explicit formula for the
derivative of a fraction. And it is straightforward to check that this
biderivation is anti-symmetric and satisfies the Jacobi identity. 
Hence $A'$ becomes a Poisson $R$-deformation
of $C'$, uniquely.

In the situation of part (2), we know (from step 1) that there is a
unique $A$-algebra homomorphism $g' : A' \to A''$
such that $h = g' \circ g$. The formula for the Poisson bracket on
$A'$ shows that $g'$ is a homomorphism of Poisson algebras. 

\medskip \noindent
Step 3. Finally we allow $R$ to be noetherian, and look at both cases
together. 
Then $R \cong \lim_{\leftarrow i}\, R_i$, and, letting
$A_i := R_i \otimes_R A$, we have $A \cong \lim_{\leftarrow i}\, A_i$.
By the previous steps for every $i$ there is an $R_i$-deformation 
$A'_i$ of $C'$. Due to uniqueness these form an inverse system, and we
take $A' := \lim_{\leftarrow i}\, A'_i$. 
By Proposition \ref{prop:12} this is an $R$-deformation of
$C'$.

Part (2) is proved similarly by nilpotent approximations. 
\end{proof}

%\cleardoublepage
\section{Sheaves of Complete Modules}
\label{sec:complete-sh}

In this section we present a few results on sheaves of $\m$-adically complete
$R$-modules on a topological space $X$. 

Suppose $\bsym{U} = \{ U_k \}_{k \in K}$ is a collection of
open sets in $X$. For $k_0, \ldots, k_m \in K$ we write
$U_{k_0, \ldots, k_m} := U_{k_0} \cap \cdots \cap U_{k_m}$.

\begin{dfn} \label{dfn:23}
Let $\mcal{N}$ be a sheaf of abelian groups on the topological space
$X$. 
\begin{enumerate}
\item An open set $U \subset X$ will be called {\em
$\mcal{N}$-acyclic} if the derived functor sheaf cohomology satisfies
$\mrm{H}^i (U, \mcal{N}) = 0$ for all $i > 0$.
\item Now suppose $\bsym{U} = \{ U_k \}_{k \in K}$ is a collection of
open sets in $X$. We say that the collection $\bsym{U}$ is
{\em $\mcal{N}$-acyclic} if all the finite intersections
$U_{k_0, \ldots, k_m}$ are $\mcal{N}$-acyclic.
\item We say that {\em there are enough $\mcal{N}$-acyclic open
sets} if for any open set $U \subset X$, and any open covering
$\bsym{U}$ of $U$, there exists an $\mcal{N}$-acyclic open covering
$\bsym{U}'$ of $U$ which refines $\bsym{U}$.
\end{enumerate}
\end{dfn}

\begin{exa} \label{exa:20}
Here are a few typical examples of a topological space $X$, and a
sheaf $\mcal{N}$, such that there are enough $\mcal{N}$-acyclic 
open sets.
\begin{enumerate}
\item $X$ is an algebraic variety over a field $\K$
(i.e.\ an integral finite type separated  $\K$-scheme), with structure
sheaf $\mcal{O}_X$, and $\mcal{N}$ is a coherent $\mcal{O}_X$-module.
Then any collection of affine open sets is $\mcal{N}$-acyclic.
\item $X$ is a complex analytic manifold, with structure sheaf 
$\mcal{O}_X$, and $\mcal{N}$ is a coherent $\mcal{O}_X$-module. Then
any collection of Stein open sets is $\mcal{N}$-acyclic.
\item $X$ is a differentiable manifold, with structure
sheaf $\mcal{O}_X$, and $\mcal{N}$ is any $\mcal{O}_X$-module. Then
any open set is $\mcal{N}$-acyclic.
\item $X$ is a differentiable manifold, and $\mcal{N}$ is a locally
constant sheaf of abelian groups. Then any sufficiently small
simply connected open set $U$ is $\mcal{N}$-acyclic. 
\end{enumerate}
\end{exa}

\begin{rem}
For the purposes of this section it suffices to require only the
vanishing of $\mrm{H}^1 (U, \mcal{N})$. But considering the examples
above, 
we see that the stronger requirement of acyclicity is not too
restrictive. Cf.\ also \cite{KS}. 
\end{rem}

Let $\K$ be a field and $(R, \m)$ a parameter $\K$-algebra.
Recall that for $i \geq 0$ we write $R_i := R / \m^{i+1}$. 

Consider a sheaf $\mcal{M}$ of $R$-modules on a topological space $X$. 
Given a ring homomorphism $R \to R'$, the sheaf 
$R' \otimes_R \mcal{M}$ is the sheaf associated to the presheaf
$U \mapsto R' \otimes_R \Gamma(U, \mcal{M})$, 
for open sets $U \subset X$. 
If $\{ \mcal{M}_i \}_{i \in \N}$ is an inverse system of sheaves on
$X$, then 
$\opn{lim}_{\leftarrow i}\, \mcal{M}_i$ is the sheaf
$U \mapsto \opn{lim}_{\leftarrow i}\, \Gamma(U, \mcal{M}_i)$.

By combining the operations above one
defines the {\em $\m$-adic completion} of the sheaf of $R$-modules $\mcal{M}$
to be 
$\what{\mcal{M}} :=
\lim_{\leftarrow i}\, (R_i \otimes_R \mcal{M})$.
The sheaf $\mcal{M}$ is called {\em $\m$-adically complete} if the
canonical sheaf homomorphism 
$\mcal{M} \to \what{\mcal{M}}$ is an isomorphism.

We define $\m^i \mcal{M}$ to be the sheaf associated to the presheaf
$U \mapsto \m^i \Gamma(U, \mcal{M})$
for open sets $U \subset X$; it is a subsheaf of $\mcal{M}$.
Next we define
$\opn{gr}^i_{\m} (\mcal{M}) := \m^i \mcal{M} / \m^{i+1} \mcal{M}$
and
$\opn{gr}_{\m} (\mcal{M}) := \boplus_{i \geq 0}
\opn{gr}^i_{\m} (\mcal{M})$. The latter is a sheaf of 
$\opn{gr}_{\m} (R)$ -modules.

The sheaf $\mcal{M}$ is called {\em flat} if for every point $x \in X$ the stalk
$\mcal{M}_x$ is a flat $R$-module. 
If $\mcal{M}$ is flat over $R$, then the canonical sheaf homomorphism
$\opn{gr}_{\m} (R) \ot \opn{gr}^0_{\m} (R)  \to \opn{gr}_{\m} (\mcal{M})$
is an isomorphism (cf.\ \cite[Theorem III.5.1]{Bo1}). Note that 
$\opn{gr}^0_{\m} (R) = \K \otimes_R \mcal{M}$.

The reason we need acyclic open sets is this:
 
\begin{thm}[{\cite[Theorem 5.6]{Ye4}}]  \label{thm:230}
Let $\K$ be a field, $(R, \m)$ a parameter $\K$-algebra, $X$ a
topological space, $U \subset X$ an open set, and $\mcal{M}$ a sheaf $R$-modules
on $X$. Define $\mcal{M}_i := R_i \otimes_R \mcal{M}$.
We assume that $\mcal{M}$ is flat over $R$ and $\m$-adically complete,
and that $U$ is an $\mcal{M}_0$-acyclic open set.
Then the $R$-module $\Gamma(U, \mcal{M})$
is flat and $\m$-adically complete, and for every $i$ the canonical
homomorphism
\[ R_i \otimes_R \Gamma(U, \mcal{M}) \to \Gamma(U, \mcal{M}_i) \]
is bijective.
\end{thm}

\begin{cor} \label{cor:240}
Let $(R, \m)$ be a parameter $\K$-algebra, $X$ a topological space, and
$\mcal{N}_0$ a sheaf of $\K$-modules on $X$. Assume that $X$ has enough
$\mcal{N}_0$-acyclic open coverings. 
\begin{enumerate}
\item The sheaf of $R$-modules $\mcal{N} := R \hot \mcal{N}_0$
is flat  and $\m$-adically complete.

\item Let $U$ be an $\mcal{N}_0$-acyclic open set of $X$. Then the canonical
homomorphism 
$R \hot \Gamma(U, \mcal{N}_0) \to \Gamma(U, \mcal{N})$
is bijective.
\end{enumerate}
\end{cor}

\begin{proof}
(1) Since $\mcal{N} = \lim_{\leftarrow i}\, (R_i \ot \mcal{N}_0)$, 
this follows from \cite[Corollary 5.10]{Ye4}.

\medskip \noindent
(2) By Theorem \ref{thm:230} we know that $\Gamma(U, \mcal{N})$
is flat and $\m$-adically complete, and 
$\K \ot_R \Gamma(U, \mcal{N}) \cong \Gamma(U, \mcal{N}_0)$.
Now use the implication (i) $\Rightarrow$ (iv) in Proposition \ref{prop:12}.
\end{proof}

\begin{cor} \label{cor:241}
In the situation of Theorem \tup{\ref{thm:230}}, let 
$M := \Gamma(U, \mcal{M})$. Then for any $j \geq 0$  the 
$R$-module $\m^j M$ is $\m$-adically complete, and 
the canonical homomorphism
$\m^j M \to \Gamma(U, \m^j \mcal{M})$ is bijective.
\end{cor}

Note that $\m^j \mcal{M}$ is usually not flat over $R$ for $j \geq 1$
(because $\m^j$ is usually not a flat $R$-module).

\begin{proof}
Let $N_0 := \Gamma(U, \mcal{M}_0)$. 
By Theorem \ref{thm:230} we know that $M$ is a flat $\m$-adically complete
$R$-module, and $\K \ot_R M \cong N_0$. Hence by Proposition \ref{prop:12} 
there is an isomorphism of $R$-modules $M \cong R \hot N_0$. Under this
isomorphism the $R$-module $\m^j M$ goes to $\m^j \hot N_0$, which 
is $\m$-adically complete. 

For any $i \in \N$ define $M_i := R_i \ot_R M$.
Consider the commutative diagram
\begin{equation} \label{eqn:36}
\UseTips \xymatrix @C=5ex @R=5ex {
0 
\ar[r]
&
\m^{j} M
\ar[r]
\ar[d]^{\alpha}
&
M
\ar[r]
\ar[d]^{=}
&
M_{j-1}
\ar[r]
\ar[d]^{\be}
& 
0
\\
0 
\ar[r]
&
\Gamma(U, \m^{j}  \mcal{M})
\ar[r]
&
\Gamma(U, \mcal{M})
\ar[r]
&
\Gamma(U, \mcal{M}_{j-1})
} \quad .
\end{equation}
The top row is trivially exact, and the bottom row is exact since $\Gamma(U, -)$
is left exact. The arrow $\be$ is bijective by Theorem \tup{\ref{thm:230}}.
We conclude that $\alpha$ is bijective.
\end{proof}

\begin{cor} \label{cor:242}
In the situation of Theorem \tup{\ref{thm:230}}, assume that $X$ has enough
$\mcal{M}_0$-acyclic open sets. Then for any $j \geq 0$ the sheaf 
of $R$-modules $\m^j \mcal{M}$ is $\m$-adically complete.
\end{cor}

\begin{proof}
Since the $\mcal{M}_0$-acyclic open sets form a basis of the topology of
$X$, it is enough to prove that the canonical homomorphism 
\[ \Gamma(V, \m^{j}  \mcal{M}) \to  
\lim_{\leftarrow i}\, \Gamma(V, R_i \ot_R \m^{j}  \mcal{M}) \]
is bijective for any $\mcal{M}_0$-acyclic open set $V$. 

Define $M := \Gamma(V, \mcal{M})$ and $M_i := R_i \ot_R M$. Now 
$R_i \ot_R \m^{j} \mcal{M} \cong \m^{j} \mcal{M}_i$, 
and by Corollary \ref{cor:241}, applied to $R_i$ instead of $R$, we know that 
$\Gamma(V, \m^{j} \mcal{M}_i) \cong \m^{j} M_i$. But 
$\m^{j} M_i \cong R_i \ot_R \m^{j} M$. Since $\m^{j} M$ is $\m$-adically
complete, it follows that 
$\m^{j} M \to \lim_{\leftarrow i}\, \m^{j} M_i$
is bijective. 
\end{proof}

An $\m$-adic system of $R$-modules on $X$ is the sheaf version of what we have
in Definition \ref{dfn:31}.

\begin{prop}[{\cite[Corollary 5.10]{Ye4}}] \label{prop:16}
Let $\{ \mcal{M}_i \}_{i \in \mbb{N}}$ be an $\m$-adic system of
$R$-modules on $X$. Assume that $X$  has enough $\mcal{M}_0$-acyclic
open coverings, and that each $\mcal{\mcal{M}}_i$ is flat over $R_i$. Then 
$\mcal{M} := \lim_{\leftarrow i}\, \mcal{M}_i$
is a flat and $\m$-adically complete sheaf of $R$-modules,
and the canonical homomorphisms 
$R_i \otimes_R \mcal{M} \to \mcal{M}_i$
are isomorphisms.
\end{prop}

Suppose $(R', \m')$ is another parameter algebra, and
$\sigma : R \to  R'$ is a homomorphism. For a sheaf $\mcal{M}$ of
$R$-modules on $X$ we let 
\begin{equation} \label{eqn:235}
R' \hatotimes{R} \mcal{M}  = \lim_{\leftarrow i}\,
(R'_i \ot_R  \mcal{M}) , 
\end{equation}
where $R'_i := R' / {\m'}^{\, i+1}$.

\begin{cor} \label{cor:245}
Let $R$, $X$ and $\mcal{M}$ be as in Theorem \tup{\ref{thm:230}}.
Assume that $X$ has enough $\mcal{M}_0$-acyclic open coverings. 
Let $R'$ be another parameter algebra, and
$R \to  R'$ a homomorphism. Define
$\mcal{M}'_i := R'_i \otimes_{R} \mcal{M}$
and
$\mcal{M}' := R' \hatotimes{R} \mcal{M}$.
Then $\mcal{M}'$ is a flat and $\m'$-adically
complete sheaf of $R'$-modules, 
the canonical homomorphisms 
$R'_i \otimes_{R'} \mcal{M}' \to \mcal{M}'_i$
are isomorphisms, and $X$ has enough $\mcal{M}'_0$-acyclic open
coverings. 
\end{cor}

\begin{proof}
Apply Proposition \ref{prop:16} to the $\m'$-adic system of $R'$-modules 
$\{ \mcal{M}'_i \}_{i \in \mbb{N}}$, noting that 
$\mcal{M}'_0 = \mcal{M}_0$ (since $R'_0 = R_0 = \K$).
\end{proof}

%\cleardoublepage
\section{Deformations of Sheaves of Algebras}
\label{sec:defs-sh}

Let $X$ be a topological space, and 
$\mcal{O}_X$ a sheaf of commutative $\K$-algebras on $X$. In this
section we define the notions of associative and Poisson $R$-deformations
of the sheaf $\mcal{O}_X$, and we establish some properties.
We work in the following setup:

\begin{setup} \label{setup:5}
$\K$ is a field; $(R, \m)$ is a parameter $\K$-algebra ; $X$ is a topological
space; and $\mcal{O}_X$ is a sheaf of commutative $\K$-algebras on $X$. The
assumption is that $X$ has enough $\mcal{O}_X$-acyclic open sets (see 
Definition \ref{dfn:23}).
\end{setup}

Recall our convention that associative algebras are unital, and commutative
algebras are associative (and unital).

\begin{dfn} \label{dfn:5}
Assume Setup \ref{setup:5}. 
An {\em associative $R$-deformation of $\mcal{O}_X$} is a sheaf
$\mcal{A}$ of flat $\m$-adically complete associative
$R$-algebras on $X$, together with an isomorphism of
sheaves of $\K$-algebras
$\psi : \K \otimes_R \mcal{A} \to \mcal{O}_X$, called an {\em augmentation}.

Suppose $\mcal{A}'$ is another associative $R$-deformation of
$\mcal{O}_X$. A {\em gauge transformation}
$g : \mcal{A} \to \mcal{A}'$
is an isomorphism of sheaves of $R$-algebras that commutes with
the augmentations to $\mcal{O}_X$.

We denote by $\cat{AssDef}(R, \mcal{O}_X)$ the groupoid whose objects are the
associative $R$-deformations of $\mcal{O}_X$, and the morphisms are the gauge
transformations.
\end{dfn}

\begin{rem}
Suppose $\opn{char} \K = 0$, $(X, \mcal{O}_X)$ is a smooth algebraic
variety over $\K$, and  $R = \K[[\hbar]]$. In our earlier
paper \cite{Ye1} we referred to an associative $R$-deformation of 
$\mcal{O}_X$ as a ``deformation quantization of $\mcal{O}_X$''. In
retrospect this name seems inappropriate,
and hence the new name used here.

Another, more substantial, change is that in 
\cite[Definition 1.6]{Ye1}
we required that the associative deformation $\mcal{A}$
shall be endowed with a differential structure. This turns out to be
redundant -- see Remark \ref{rem:defs-sh.110}.
\end{rem}

\begin{dfn} \label{dfn:12}
Assume Setup \ref{setup:5}. 
We view $\mcal{O}_X$ as a sheaf of Poisson $\K$-algebras
with the zero bracket.
A {\em Poisson $R$-deformation of $\mcal{O}_X$} is a sheaf
$\mcal{A}$ of flat $\m$-adically complete commutative
Poisson $R$-algebras on $X$, together with an isomorphism of
Poisson $\K$-algebras
$\psi : \K \otimes_R \mcal{A} \to \mcal{O}_X$,
called an {\em augmentation}.

Suppose $\mcal{A}'$ is another Poisson $R$-deformation of
$\mcal{O}_X$. A {\em gauge transformation}
$g : \mcal{A} \to \mcal{A}'$
is an isomorphism of sheaves of Poisson $R$-algebras that commutes
with the augmentations to $\mcal{O}_X$.

We denote by $\cat{PoisDef}(R, \mcal{O}_X)$ the groupoid whose objects are
the Poisson $R$-deformations of $\mcal{O}_X$, and the morphisms are the
gauge transformations.
\end{dfn}

\begin{prop} \label{prop:233}
Let $\mcal{A}$ be a Poisson \tup{(}resp.\ associative\tup{)}
$R$-deformation of $\mcal{O}_X$, and let $U$ be an
$\mcal{O}_X$-acyclic open set of $X$. 
Then $A := \Gamma(U, \mcal{A})$ is a 
Poisson \tup{(}resp.\ associative\tup{)} $R$-deformation of 
$C := \Gamma(U, \mcal{O}_X)$.
\end{prop}

\begin{proof}
By Theorem \ref{thm:230}, $A$ is a flat $\m$-adically complete $R$-algebra, and 
the homomorphism $\K \ot_R A \to C$ is bijective. 
\end{proof}

\begin{prop} \label{prop:23}
Let $\mcal{A}$ be a Poisson \tup{(}resp.\ associative\tup{)}
$R$-deformation of $\mcal{O}_X$,  let $R'$ be another
parameter $\K$-algebra, and let $\sigma : R \to  R'$ a $\K$-algebra
homomorphism. Define 
$\mcal{A}' := R' \hatotimes{R} \mcal{A}$.
Then $\mcal{A'}$ is a Poisson \tup{(}resp.\ associative\tup{)}
$R'$-deformation of $\mcal{O}_X$.
\end{prop}

\begin{proof}
The sheaf $\mcal{A'}$ has an induced $R'$-bilinear Poisson bracket 
(resp.\ multiplication). By Corollary \ref{cor:245} the sheaf of
$R'$-modules $\mcal{A'}$ is flat and $\m'$-adically complete, and the 
the canonical homomorphism $\K \ot_{R'} \mcal{A'} \to \OX$ is an isomorphism.
\end{proof}

Here is a converse to Proposition \ref{prop:233}, in the affine
algebro-geometric setting. 

\begin{thm} \label{thm:defs-sh.112}
Let $X$ be a smooth algebraic variety over $\K$, let $U$ be an affine open set
of $X$, and let $C := \Gamma(X, \mcal{O}_X)$. 
\begin{enumerate}
\item Let $A$ be a Poisson \tup{(}resp.\ associative\tup{)}
$R$-deformation of $C$. Then there exists a 
Poisson \tup{(}resp.\ associative\tup{)} $R$-deformation 
$\mcal{A}$ of $\mcal{O}_U$, together with a 
gauge transformation of deformations
$g : A \to \Gamma(U, \mcal{A})$. 

\item Let $\mcal{A}$ and $\mcal{A}'$ be Poisson \tup{(}resp.\
associative\tup{)}
$R$-deformations of $\mcal{O}_U$, and let
$h : \Gamma(U, \mcal{A}) \to \Gamma(U, \mcal{A}')$
be a gauge transformation of deformations. Then there is a unique 
gauge transformation of deformations
$\til{h} : \mcal{A} \to \mcal{A}'$
such that $\Gamma(U, \til{h}) = h$.
\end{enumerate}
\end{thm}

Note that part (2) implies that the pair $(\mcal{A}, g)$ of part (1)
is unique
up to a unique gauge transformation.

\begin{proof}
The proof is in several steps.

\medskip \noindent
Step 1. Assume $R$ is artinian. 
For an element $s \in C$ we
denote by $U_{s}$ the affine open set
$\{ x \in U \mid s(x) \neq 0 \}$; and we call it a principal open set.
Note that $\Gamma(U_{s}, \mcal{O}_U) \cong C_s$.
By Theorem \ref{thm:12} there is a deformation $A_s$ of $C_s$, unique
up to a unique gauge transformation. 

Now suppose $t$ is another element of $C$, and
$U_{t} \subset U_{s}$.
Then we have $\K$-algebra homomorphisms
$C \to C_s \to C_t$. Again by Theorem \ref{thm:12}, there is a unique
homomorphism of Poisson (resp.\ associative) $R$-algebras 
$A_s \to A_t$ that's compatible with the homomorphisms from $A$. 

By this process we obtain a presheaf of  Poisson (resp.\ associative)
$R$-algebras on the principal open sets of $U$. 
Since these open sets are a basis of the topology of $U$, 
according to \cite[Chapter 0$_{\text{I}}$, Section 3.2.1]{EGA-I} 
this gives rise to a presheaf $\mcal{A}$ of Poisson (resp.\ associative)
$R$-algebras on $U$, such that 
$\Gamma(U_s, \mcal{A}) = A_s$ for every principal open set $U_s$.

In order to show that $\mcal{A}$ is a sheaf, it suffices (by 
\cite[Chapter 0$_{\text{I}}$, Section 3.2.2]{EGA-I})
to prove that for any principal open set $U_s$, and any finite 
covering $U_s = \bigcup_{k \in K} U_{t_k}$ of $U_s$ by principal open sets,
the sequence of $R$-modules 
\begin{equation} \label{eqn:defs-sh.112}
0 \to A_s \to  \prod_{k_0 \in K} A_{t_{k_0}} \to
\prod_{k_0, k_1 \in K} A_{t_{k_0} t_{k_1}} 
\end{equation}
is exact. (Note that $U_{t_{k_0}} \cap U_{t_{k_1}} = U_{t_{k_0} t_{k_1}}$.)  
Let us write $R_i := R / \m^{i+1}$ as usual; so $R_0 = \K$, and 
$R_i = R$ for sufficiently large $i$. We will prove that the sequence 
gotten from (\ref{eqn:defs-sh.112}) by the operation $R_i \ot_R -$ is exact,
by induction on $i$. 
For $i = 0$ we have $R_0 \ot_R A_t = C_t = \Gamma(U_t, \mcal{O}_X)$ for any
$t \in C$; so the exactness of (\ref{eqn:defs-sh.112}) for $R_0$ is true
because $\mcal{O}_X$ is a sheaf.
Now assume $i > 0$, and the sequence is exact for all $R_{j}$, $j < i$. 
There is an exact sequence of $R$-modules 
\begin{equation} \label{eqn:defs-sh.113}
 0  \to \m^i R_i \to R_i \to R_{i-1} \to 0 ,
\end{equation}
and $\m^i R_i \cong \m^i / \m^{i+1}$, so this is an
$R_0$-module. All the $R$-modules in (\ref{eqn:defs-sh.112}) are flat, and
hence when we tensor this sequence  with the exact sequence 
(\ref{eqn:defs-sh.113}), written vertically, we get a commutative diagram with
exact columns. By assumption the rows corresponding to $\m^i R_i$ and $R_{i-1}$
are exact; and therefore the row in between, the one corresponding to $R_i$, is
also exact.

\medskip \noindent
Step 2. $R$ is still artinian. 
Let  $\mcal{A}$ be the sheaf of algebras from the first step.
Take a point $x \in U$. Then the stalk
$\mcal{A}_x \cong \opn{lim}_{\to} A_s$,
the limit taken over the elements $s \in C$ such that 
$x \in U_{s}$. This shows that $\mcal{A}_x$ is a flat $R$-module;
and hence the sheaf $\mcal{A}$ is flat. The construction of $\mcal{A}$ endows
it with an augmentation to $\mcal{O}_X$. We conclude that 
$\mcal{A}$ is an $R$-deformation of $\mcal{O}_U$.

Now look at the $R$-algebra homomorphism
$g : A \to \Gamma(U, \mcal{A})$.
Since both are flat $R$-algebras augmented to $C$, it follows that
$g$ is an isomorphism.

\medskip \noindent
Step 3. Here we handle part (2), still with $R$ artinian.
Suppose $\mcal{A}$ and $\mcal{A}'$ are
two  $R$-deformations of $\mcal{O}_U$.
Write $A := \Gamma(U, \mcal{A})$ and 
$A' := \Gamma(U, \mcal{A}')$.
We are given a gauge transformation 
$h : A \to A'$. Take $s \in C$. Since $C \to C_s$ is a principal
localization, and both
$\Gamma(U_s, \mcal{A})$ and $\Gamma(U_s, \mcal{A}')$ are
$R$-deformations of $C_s$, Theorem \ref{thm:12}(2) says that there is a
unique gauge transformation
$\Gamma(U_s, \mcal{A}) \iso \Gamma(U_s, \mcal{A}')$
that's compatible with the homomorphisms from $A$.
In this way we obtain an isomorphism of sheaves 
$\til{h} : \mcal{A} \to \mcal{A}'$ extending $h$; and it is unique.

\medskip \noindent
Step 4. Finally we allow $R$ to be noetherian. 
Then $R \cong \lim_{\leftarrow i}\, R_i$, and, letting
$A_i := R_i \otimes_R A$, we have $A \cong \lim_{\leftarrow i}\, A_i$.
By the previous steps for every $i$ there is an $R_i$-deformation 
$\mcal{A}_i$. Due to uniqueness these form an inverse system, and we
take $\mcal{A} := \lim_{\leftarrow i}\, \mcal{A}_i$. 
By Proposition \ref{prop:16} this is an $R$-deformation of $\mcal{O}_U$.

Part (2) is also proved by nilpotent approximation.
\end{proof}

\begin{cor} \label{cor:243}
Let $X$ be a smooth algebraic variety over $\K$, let $U$ be an affine open set
of $X$, and let $C := \Gamma(X, \mcal{O}_X)$. Then the morphisms of groupoids
\[ \Gamma(U, -) : \cat{AssDef}(R, \mcal{O}_X) \to 
\cat{AssDef}(R, C) \]
and 
\[ \Gamma(U, -) : \cat{PoisDef}(R, \mcal{O}_X) \to 
\cat{PoisDef}(R, C) \]
are equivalences.
\end{cor}

\begin{proof}
According to Theorem \ref{thm:defs-sh.112}(1) we have essential surjectivity on
objects. And Theorem \ref{thm:defs-sh.112}(2) says that the functors 
$\Gamma(U, -)$ are fully faithful. 
\end{proof}

%\cleardoublepage
\section{The Crossed Groupoid of Deformations}

Here we assume this setup:

\begin{setup} \label{setup:230}
$\K$ is a field of characteristic $0$; $(R, \m)$ is a parameter $\K$-algebra ;
$X$ is a topological space; and $\mcal{O}_X$ is a sheaf of commutative
$\K$-algebras on $X$. The assumption is that $X$ has enough $\mcal{O}_X$-acyclic
open sets.
\end{setup}

This is setup \ref{setup:5}, plus the condition $\opn{char} \K = 0$. 

Let us say a few words on {\em sheaves of pronilpotent groups}. 
Consider a sheaf of groups $\mcal{G}$ on $X$. 
A {\em central filtration} of $\mcal{G}$ is a descending
filtration $\{ \mcal{N}^j \}_{j \in \N}$ by normal subgroups,
such that $\mcal{N}^0 = \mcal{G}$, 
$\bigcap_j \mcal{N}^j = 1$, and  
$\mcal{N}^j / \mcal{N}^{j+1}$ is central in 
$\mcal{G} / \mcal{N}^{j+1}$ for every $j$. Thus 
$\mcal{G} / \mcal{N}^{j}$ is nilpotent. The sheaf $\mcal{G}$ is said to be {\em
complete} with respect to the filtration 
$\{ \mcal{N}^j \}_{j \in \N}$ 
if the canonical group homomorphism 
$\mcal{G} \to \lim_{\leftarrow j} \, (\mcal{G} / \mcal{N}^j)$
is an isomorphism. The sheaf of groups $\mcal{G}$ is called pronilpotent if it
is complete with respect to some central filtration. 

Next consider a sheaf $\mcal{A}$ of Lie $R$-algebras on $X$, such that for
every $j \in \N$ the sheaf $\mcal{B}^j := \m^j \mcal{A}$ is $\m$-adically
complete, and such that $[\mcal{A}, \mcal{A}] \subset \m \mcal{A}$.
Now $\mcal{B}^j / \mcal{B}^{j+i}$
is sheaf of nilpotent Lie $\K$-algebras,
and so there is an associated sheaf of nilpotent groups 
$\exp(\mcal{B}^j / \mcal{B}^{j+i})$, and
an isomorphism of sheaves of sets 
$\exp_{\mcal{B}^j / \mcal{B}^{j+i}} : \mcal{B}^j / \mcal{B}^{j+i} 
\to \exp (\mcal{B}^j / \mcal{B}^{j+i})$.
Passing to the inverse limit we obtain a sheaf of groups 
$\exp(\mcal{B}^j) := \lim_{\leftarrow i}\, (\mcal{B}^j / \mcal{B}^{j+i})$,
and an isomorphism of sheaves of sets 
$\exp_{\mcal{B}^j} : \mcal{B}^j \to \exp(\mcal{B}^j)$.
The sheaf of groups $\exp(\m^j \mcal{A}) = \exp(\mcal{B}^j)$ is pronilpotent;
indeed, 
$\{ \exp(\mcal{B}^{j + i}) \}_{i \in \N}$ is a central filtration of 
$\exp(\mcal{B}^j)$, and $\exp(\mcal{B}^j)$ is complete with respect to this
filtration. 
This construction is functorial: if $\mcal{A}'$ is another such sheaf of
Lie $R$-algebras, and $\phi : \mcal{A} \to \mcal{A}'$
is an $R$-linear Lie homomorphism, then there is a group homomorphism 
$\exp(\phi) : \exp(\m^j \mcal{A}) \to \exp(\m^j \mcal{A}')$, and 
$\exp(\phi) \circ \exp_{\m^j \mcal{A}} = \exp_{\m^j \mcal{A}'} \circ \, \phi$.

Let $\al \in \mcal{A}$ be a local section, defined on some open set $U$. There
is an $R$-linear endomorphism $\opn{ad}_{\mcal{A}}(\al)$ of $\mcal{A}|_U$ whose
formula is 
\begin{equation} \label{eqn:236}
\opn{ad}_{\mcal{A}}(\al)(\al') := [\al, \al'] ,
\end{equation}
where $[-,-]$ is the Lie bracket of $\mcal{A}$. 
Let us denote by $\End_R(\mcal{A})$ the sheaf of $R$-module endomorphisms of
$\AA$. Then 
$\opn{ad}_{\mcal{A}} : \mcal{A} \to \End_R(\mcal{A})$
is a Lie algebra homomorphism. In this way we get a homomorphism of sheaves of
groups 
\begin{equation} \label{eqn:237}
\begin{aligned}
& \opn{Ad}_{\mcal{A}} : \exp(\mcal{A}) \to \Aut_R(\mcal{A}) , 
\\ 
& \opn{Ad}_{\mcal{A}}(\exp_{\mcal{A}}(\al)) :=
\exp(\opn{ad}_{\mcal{A}}(\al)) =
\sum_{i \geq 0} \, \smfrac{1}{i!} \,
\underset{i}{\underbrace{\opn{ad}_{\mcal{A}}(\al) \circ \cdots \circ 
\opn{ad}_{\mcal{A}}(\al)}} . 
\end{aligned} 
\end{equation}
Cf.\ \cite[Section 2.3]{Hu}.
Note that this series converges $\m$-adically, since 
$\opn{ad}_{\mcal{A}}(\al)^i(\mcal{A}) \subset \m^i \mcal{A}$.

Let $\mcal{A}$ be an associative (resp.\ Poisson) $R$-deformation of
$\mcal{O}_X$. Then $\mcal{A}$ has an $R$-linear Lie bracket on it; in the
associative case it is the commutator bracket
\[ [a_1, a_2] := a_1 \star a_2 - a_2 \star a_1 , \]
and in the Poisson case it is the Poisson bracket.

\begin{prop} \label{prop:237}
Let $\mcal{A}$ be an associative \tup{(}resp.\ Poisson\tup{)} $R$-deformation
of $\mcal{O}_X$. For a section $\al \in \Gamma(U, \m \mcal{A})$ the $R$-linear
automorphism $g := \exp(\opn{ad}_{\mcal{A}}(\al))$ of $\mcal{A}|_U$  from
\tup{(\ref{eqn:237})}  is a gauge transformation of $R$-deformations 
of $\OO_U$.
\end{prop}

\begin{proof}
In the associative case $\opn{ad}_{\mcal{A}}(\al)$ is a derivation of the 
algebra $\mcal{A}|_U$; so according to \cite[Section 2.3]{Hu}, $g$ is 
an automorphism the algebra $\mcal{A}|_U$.

In the Poisson case $\opn{ad}_{\mcal{A}}(\al)$ is a
derivation both of the commutative algebra $\mcal{A}|_U$ and of its Poisson
bracket. Hence $g$ is a Poisson automorphism of $\mcal{A}|_U$.

Since $\opn{ad}_{\mcal{A}}(\al) \equiv 0$ modulo $\m$, it follows that $g$
commutes with the augmentation $\mcal{A}|_U \to \mcal{O}_U$. 
\end{proof}

\begin{dfn} \label{dfn:240}
Let $\mcal{A}$ be an associative (resp.\ Poisson) $R$-deformation of
$\mcal{O}_X$. 
\begin{enumerate}
\item Define the sheaf of groups  
\[ \opn{IG}(\mcal{A}) := \exp(\m \mcal{A}) . \]
It is called the {\em sheaf of inner gauge group} of $\mcal{A}$.

\item For a local section $a = \exp(\al) \in \Gamma(U, \opn{IG}(\mcal{A}))$ we
define the
gauge transformation $\opn{ig}(a)$ of $\mcal{A}|_U$ to be 
\[ \opn{ig}(a) := \exp(\opn{ad}_{\mcal{A}}(\al)) . \]
\end{enumerate}
\end{dfn}

If $g : \mcal{A} \to \mcal{A}'$ is a gauge transformation, then there is an
induced isomorphism of sheaves of Lie algebras 
$g : \m \mcal{A} \to \m \mcal{A}'$,
and, by taking exponentials, an induced isomorphism of sheaves of groups 
\begin{equation} \label{eqn:240}
\opn{IG}(g) : \opn{IG}(\mcal{A}) \to \opn{IG}(\mcal{A}') . 
\end{equation}

\begin{prop} \label{prop:238}
The groupoid $\cat{AssDef}(R, \mcal{O}_X)$ 
\tup{(}resp.\ $\cat{PoisDef}(R, \mcal{O}_X)$\tup{)} is the $1$-truncation
of a crossed group\-oid, where\tup{:}
\begin{itemize}
\item The $2$-morph\-isms are the inner gauge transformations, namely the
elements of the groups $\Gamma(X, \opn{IG}(\mcal{A}))$.

\item The twisting by a gauge transformation $g : \mcal{A} \to \mcal{A}'$ is
the group isomorphism $\opn{IG}(g)$ from  \tup{(\ref{eqn:240})}.

\item The feedback $\opn{D}(a)$, for $a \in \Gamma(X, \opn{IG}(\mcal{A}))$,
is the group isomorphism $\opn{ig}(a)$ from Definition \tup{\ref{dfn:240}(2)}.
\end{itemize}
\end{prop}

\begin{proof}
We must verify the conditions in Definition \ref{dfn:cosim.101}. 
Take a  gauge transformation $g : \mcal{A} \to \mcal{A}'$.
The diagram of Lie algebra homomorphisms 
\[ \UseTips \xymatrix @C=7ex @R=5ex {
\m \mcal{A}
\ar[d]_{\opn{ad}}
\ar[r]^{g}
&
\m \mcal{A}'
\ar[d]^{\opn{ad}}
\\
\End(\mcal{A})
\ar[r]^{\opn{Ad}(g)}
&
\End(\mcal{A}')
} \]
is commutative. Taking the exponentials we see that 
$\opn{ig} \circ \opn{IG}(g) = \opn{Ad}(g) \circ \opn{ig}$.

Next let us look at $a = \exp(\al) \in \opn{IG}(\mcal{A})$.
Then $\opn{ig}(a) = \exp(\opn{ad}(\al)) : \mcal{A} \to \mcal{A}$,
and 
$\opn{IG}(\opn{ig}(a)) : \opn{IG}(\mcal{A}) \to \opn{IG}(\mcal{A})$
is the restriction of $\exp(\opn{ad}(\al))$ to $\m \mcal{A}$. 
This says that $\opn{IG}(\opn{ig}(a))$ is conjugation by $a$ in the group 
$\opn{IG}(\mcal{A})$.
\end{proof}

\begin{prop} \label{prop:239}
Let $U \subset X$ be an open set and $R \to R'$ a homomorphism of parameter
algebras. Then the formula 
$\mcal{A} \mapsto (R' \hatotimes{R} \mcal{A})|_U$
gives rise to morphisms of crossed groupoids
\[ \cat{AssDef}(R, \mcal{O}_X) \to \cat{AssDef}(R', \mcal{O}_U) \]
and
\[ \cat{PoisDef}(R, \mcal{O}_X) \to \cat{PoisDef}(R', \mcal{O}_U) . \]
\end{prop}

\begin{proof}
For restriction to $U'$ this is clear. As for $R \to R'$, this is Proposition
\ref{prop:23}. 
\end{proof}

\begin{prop} \label{prop:234}
Let $\mcal{A}$ be an associative $R$-deformation of $\OX$, with 
augmentation $\psi : \mcal{A} \to \OX$.  
There is a canonical isomorphism of sheaves of groups 
\[ \opn{IG}(\mcal{A}) \cong 
\opn{Ker} \big( \psi : \mcal{A}^{\times} \to \mcal{O}_X^{\times} \big) . \]
Under this isomorphism the inner action $\opn{ig}(a)$ 
is sent to the conjugation action by the invertible element $a$. 
\end{prop}

\begin{proof}
Let $U \subset X$ be an affine open set. According to Proposition
\ref{prop:233}, $A := \Gamma(U, \mcal{A})$ is an $R$-deformation of 
$C := \Gamma(U, \OX)$. Also, by Corollary \ref{cor:241} we have 
$\Gamma(U, \opn{IG}(\mcal{A})) \cong \exp(\m A)$ as groups. 
By Proposition \ref{prop:12} there is an isomorphism 
$A \iso R \hot C$ of augmented $R$-modules such that $1_A \mapsto 1_R \ot 1_C$.
Thus 
$\opn{Ker} \big( \psi : A^{\times} \to C^{\times} \big) = 1 + \m A$.

Since the Lie bracket on $\m A$ is the associative commutator, it follows that 
\[ \exp(\al) \mapsto \bosum_{i \geq 0} \smfrac{1}{i!} \,
\underset{i}{\underbrace{\al \star \cdots \star \al}}  \]
is a group isomorphism from the abstract pronilpotent group 
$\exp(\m A)$ to the multiplicative group 
$1 + \m A \subset A^{\times}$.
The action $\exp(\opn{ad}(\al))$ of $\exp(\al) \in \exp(\m A)$ goes to the 
conjugation action $\opn{Ad}(\exp(\al))$ in the ring $A$.

Finally, since the affine open sets are a basis of the topology, we get
the statement on the sheaf level.
\end{proof}

All the above holds of course for $R$-deformations of a commutative
$\K$-algebra $C$. Thus there are crossed groupoids
$\cat{AssDef}(R, C)$ and $\cat{PoisDef}(R, C)$,
where the $1$-morphisms are the gauge transformations, and 
the $2$-morphisms are the elements of the
groups $\opn{IG}(A) = \exp(\m A)$.

%\cleardoublepage
\section{The Deligne Crossed Groupoid}
\label{sec:deligne}

Here the base field $\K$ has characteristic $0$.
Let $\mfrak{g} = \boplus_{p \in \mbb{Z}}\, \mfrak{g}^p$
be a DG (differential graded) Lie algebra over $\K$ , with
differential $\d$ and Lie bracket $[-,-]$.
We define the pronilpotent DG Lie $R$-algebra
$\m \hot \g$ as follows. For every $p$ we let
$\m \hot \g^p := \what{\m \ot \g^p}$, the $\m$-adic completion of the
$R$-module $\m \ot \g^p$ (cf.\ Proposition \ref{prop:12}). 
Then 
$\m \hot \g := \bigoplus_{p \in \Z} \, \m \hot \g^p$.
The differential $\d$ and Lie bracket $[-,-]$ of $\m \hot \g$ are the
$R$-mutlilinear extensions of those of $\g$.

In degree $0$ we have a pronilpotent Lie algebra $\m \hot \g^0$, and we
denote by \lb $\opn{exp}(\m \hot \g^0)$ the associated
pronilpotent group. 
There is a canonical bijection of sets 
$\exp : \m \hot \g^0 \to \opn{exp}(\m \hot \g^0)$.
We call $\opn{exp}(\m \hot \g^0)$ the {\em gauge group} of 
$\m \hot \g$. 

As usual, for any element 
$\gamma \in \m \hot \g$, we denote by
$\opn{ad}(\gamma)$ the $R$-linear operator on $\m \hot \g$
with formula 
$\opn{ad}(\gamma)(\beta) := [\gamma, \beta]$.
If $\gamma \in \m \hot \g^0$, 
and we write 
$g := \opn{exp}(\gamma) \in \opn{exp}(\m \hot \g^0)$,
then we obtain an $R$-linear automorphism 
$\opn{Ad}(g) := \opn{exp}(\opn{ad}(\gamma))$
of the graded Lie algebra $\m \hot \g$ (that usually does not commute with
$\d$).

An {\em MC element} in  $\m \hot \g$
is an element
$\om \in \m \hot \g^1$
which satisfies the {\em Maurer-Cartan equation}
$\d(\om) + \smfrac{1}{2} [\om, \om] = 0$.
We denote by 
$\mrm{MC}(\m \hot \g)$ the set of MC elements.

The Lie algebra $\m \hot \g^0$
acts on the $R$-module $\m \hot \g^1$
also by the affine transformations
\begin{equation} \label{eqn:extras.5}
\opn{af}(\gamma)(\om) := \d(\gamma) - \opn{ad}(\gamma)(\om) = 
\d(\gamma) - [\gamma, \om] ,
\end{equation}
for $\gamma \in \m \hot \g^0$ and $\om \in \m \hot \g^1$.
This action integrates to an affine transformation
$\opn{Af}(g) := \opn{exp}(\opn{af}(\ga))$
of $\m \hot \g^1$, for $g := \exp(\ga)$. 
The action $\opn{Af}$ of the group $\exp(\m \hot \g^0)$ on $\m \hot \g^1$
preserves the set $\mrm{MC}(\m \hot \g)$, and we write
$\ol{\mrm{MC}}(\m \hot \g)$
for the quotient set by this action. 

Suppose $\mfrak{h}$ is another DG Lie algebra, and 
$\phi : \g \to \mfrak{h}$
is a homomorphism of DG Lie algebras. There is an induced
$R$-linear homomorphism 
$\bsym{1}_{\m} \hot \phi : \m \hot \g \to \m \hot \mfrak{h}$
of DG Lie algebras, and an induced function
\[ \mrm{MC}(\bsym{1}_{\m} \hot \phi) : \mrm{MC}(\m \hot \g) \to 
\mrm{MC}(\m \hot \mfrak{h}) . \]
If $\phi$ is a quasi-isomorphism then so is
$\bsym{1}_{\m} \hot \phi$, and on gauge equivalence classes of MC elements we
get a bijection
\begin{equation} \label{eqn:29}
\ol{\mrm{MC}}(\bsym{1}_{\m} \hot \phi) :
\ol{\mrm{MC}}(\m \hot \g) \to 
\ol{\mrm{MC}}(\m \hot \mfrak{h}) . 
\end{equation}
This is \cite[Theorem 4.2]{Ye5}. (The nilpotent case, i.e.\ $R$ artinian, was
known before of course; see \cite{GM} and \cite[Section I.3.4]{CKTB}.)

For an element $\om \in \mrm{MC}(\m \hot \g)$ we let
$\d_{\om} := \d + \opn{ad}(\om)$,
which is a derivation of degree $1$ and square $0$ of the graded Lie algebra
$\m \hot \g$. Note that for $\alpha \in \m \hot \g$ one has
$\d_{\om}(\alpha) = \d(\alpha) + [\om, \alpha]$,
and for $\gamma \in \m \hot \g^0$ one has
$\d_{\om}(\gamma) = \opn{af}(\gamma)(\om)$.

\begin{dfn} \label{dfn:dglie.101}
We say $\g = \boplus_{p \in \mbb{Z}}\, \mfrak{g}^p$ is a {\em quantum type DG
Lie algebra} if $\g^p = 0$ for all $p < -1$.
\end{dfn}

Suppose $\g$ is a quantum type DG Lie algebra. Take any 
$\om \in \mrm{MC}(\m \hot \g)$. The formula
$[\alpha_1, \alpha_2]_{\om} := 
\bigl[ \d_{\om}(\alpha_1), \alpha_2 \bigr]$
defines an $R$-bilinear Lie bracket on $\m \hot \g^{-1}$. 
We denote the resulting pronilpotent Lie algebra by
$(\m \hot \g^{-1})_{\om}$,
and the associated pronilpotent group is denoted by 
\begin{equation} \label{eqn:200}
N_{\om} := \exp(\m \hot \g^{-1})_{\om} .
\end{equation}
The function
$\d_{\om} : (\m \hot \g^{-1})_{\om} \to \m \hot \g^0$
is an $R$-linear Lie algebra homomorphism,  so it induces a group homomorphism 
\begin{equation} \label{eqn:Lie-desc.103}
\opn{D}_{\om} : N_{\om} \to \exp(\m \hot \g^0) \ , \
\opn{D}_{\om} := \exp(\d_{\om}) .
\end{equation}

Now take $g \in \exp(\m \hot \g^0)$, and let 
$\om' := \opn{Af}(g)(\om) \in \opn{MC}(\m \hot \g)$.
According to \cite[Corollary 6.9]{Ye5} there is a group isomorphism 
\begin{equation} \label{eqn:Lie-desc.104}
\opn{Ad}(g) : N_{\om} \iso N_{\om'} ,
\end{equation}
which is functorial in $g$, and the diagram 
\[ \UseTips \xymatrix @C=5ex @R=5ex {
N_{\om}
\ar[d]_{\opn{Ad}(g)}
\ar[r]^(0.34){\opn{D}_{\om}}
&
\opn{exp}(\m \hot \g^0)
\ar[d]^{\opn{Ad}(g)}
\\
N_{\om'}
\ar[r]^(0.34){\opn{D}_{\om'}}
&
\opn{exp}(\m \hot \g^0)
} \]
is commutative. By definition of the bracket $[-,-]_{\om}$, the adjoint
action in the Lie algebra $(\m \hot \g^{-1})_{\om}$ is
$\opn{ad}_{}(\al_1)(\al_2) = \opn{ad}(\d_{\om}(\al_1))(\al_2)$; 
hence, by exponentiating this equation, we see that conjugation in the group 
$N_{\om}$ is  
$\opn{Ad}_{N_{\om}}(a_1)(a_2) = \opn{Ad} (\opn{D}_{\om}(a_1))(a_2)$
for $a_i \in N_{\om}$.

Crossed groupoids were introduced in Definition \ref{dfn:cosim.101}.
The considerations above justify the next definition. 

\begin{dfn} \label{dfn:Lie-desc.101}
Let $\K$ be a field of characteristic $0$, let $\g$ be a quantum type DG Lie
$\K$-algebra, and let $(R, \m)$ be a parameter $\K$-algebra.
The {\em Deligne crossed groupoid} 
is the crossed groupoid $\opn{Del}(\g, R)$ with these
components:
\begin{itemize}
\item The groupoid $\opn{Del}_1(\g, R)$ is the transformation groupoid 
associated to the action $\opn{Af}$ of the gauge group $\exp(\m \hot \g^0)$
on the set $\opn{MC}(\m \hot \g)$.
(This is the usual Deligne groupoid of $\m \hot \g$.)

\item  The groupoid $\opn{Del}_2(\g, R)$ is the totally disconnected
groupoid with set of objects $\opn{MC}(\m \hot \g)$, and automorphism groups 
$N_{\om}$ from formula (\ref{eqn:200}).

\item The twisting $\opn{Ad}$ is the group isomorphism in formula 
(\ref{eqn:Lie-desc.104}).

\item The feedback $\opn{D}$ is the group homomorphism in formula 
(\ref{eqn:Lie-desc.103}).
\end{itemize}
\end{dfn}

It is obvious from the construction that $\opn{Del}(\g, R)$ is functorial in
both $\g$ and $R$. 

\begin{rem}
In the nilpotent case (i.e.\ when the ring $R$ is artinian) the Deligne crossed
groupoid was introduced by Deligne in a letter to Breen from 1994; see also 
\cite{Ge}. 

A more general construction than Definition \ref{dfn:Lie-desc.101} (for
unbounded DG Lie algebras) can be found in \cite{Ye5}. 
\end{rem}

%\cleardoublepage
\section{DG Lie Algebras and Deformations}
\label{sec:dglie}

In this section we recall the role of DG Lie algebras in deformation
quantization, and prove a few basic results. 
For more details see \cite[Section 1]{GM}, \cite[Section 2.3]{Ge}, 
\cite[Section I.3]{CKTB} or \cite[Section 1]{Ye5}. 
We assume here that the base field $\K$ has characteristic $0$. 

Let $V$ be a $\K$-module.
Then $R \hot V$ is an $\m$-adically complete $R$-module, with 
an augmentation $R \hot V \to V$ induced from the
augmentation $R \to \K$.
By gauge transformation of $R \hot V$ we mean an $R$-linear automorphism 
that commutes with the augmentation. 

 For $\ga \in \m \hot \opn{End}(V)$ we let 
\begin{equation} \label{eqn:dglie.118}
\opn{exp}(\ga) = 
\sum_{i \geq 0} \, \smfrac{1}{i!} \,
\underset{i}{\underbrace{\gamma \circ \cdots \circ \gamma}}
\in \opn{End}_R(R \hot V) . 
\end{equation}
This is a gauge transformation of $R \hot V$.

\begin{lem} \label{lem:dglie.111}
Let $V$ be a $\K$-module. Every gauge transformation $g$ of the augmented
$R$-module $R \hot V$ is uniquely of the form 
$g = \exp(\ga)$, for $\ga \in \m \hot \opn{End}(V)$.
\end{lem}

\begin{proof}
Take any gauge transformation $g$ of the augmented $R$-module $R \hot V$.
So $g : R \hot V \to R \hot V$ lifts $\bsym{1}_V$, the identity of $V$.
According to Proposition \ref{prop:12} we have $g = \bsym{1}_V + \phi$,
where $\phi : V \to \m \hot V$ is an arbitrary $\K$-linear
homomorphism. Since $\m_i := \m / \m^{i+1}$ is a finitely generated
$\K$-module, we have
\[ \opn{Hom}(V, \m \hot V) \cong
\lim_{\leftarrow i}\, \opn{Hom}(V, \m_i \ot V) \cong
\lim_{\leftarrow i}\, (\m_i \ot \opn{End}(V)) \cong
\m \hot \opn{End}(V) . \]
We see that $\phi \in \m \hot \opn{End}(V)$.
But then $g = \exp(\ga)$ for a unique 
$\ga \in \m \hot \opn{End}(V)$, namely 
$\ga = \log(\bsym{1}_V + \phi)$.
\end{proof}

Sometimes it is convenient to have a more explicit (but less
canonical) way of describing the $R$-module
$\m \hot V$. This is done via choice of {\em filtered $\K$-basis} of $\m$.

A filtered $\K$-basis of a finitely generated
$R$-module $M$ is a sequence
$\{ m_j \}_{j \geq 0}$ of elements of $M$ (finite if $M$
has finite length, and countable otherwise) whose symbols form a
$\K$-basis of the graded $\K$-module 
$\opn{gr}_{\m} (M)$.
It is easy to find such bases: simply choose a $\K$-basis of 
$\opn{gr}_{\m} (M)$ consisting of homogeneous elements, and lift it to
$M$. Once such a filtered basis is chosen, any element $m \in M$ has
a unique convergent power series expansion 
$m = \sum_{j \geq 0} \lambda_j m_j$, with $\lambda_j \in \K$. 

Let us choose a filtered $\K$-basis
$\{ r_j \}_{j \geq 0}$ of $R$, such that $r_0 = 1$. 
Then the sequence
$\{ r_j \}_{j \geq 1}$ is a filtered $\K$-basis of $\m$.

\begin{exa} \label{exa:14}
For the power series ring $R = \K[[\hbar]]$ the obvious filtered basis is 
$r_j := \hbar^j$.
\end{exa}

\begin{setup} \label{setup:3}
$\K$ is a field of characteristic $0$; $(R, \m)$ is a 
parameter $\K$-algebra (see Definition \ref{dfn:13}); 
and $C$ is a smooth integral commutative $\K$-algebra (i.e.\ $\opn{Spec} C$
is a smooth affine algebraic variety over $\K$).
\end{setup}

For Poisson deformations the relevant DG Lie algebra is the algebra
of {\em polyderivations}
\[ \mcal{T}_{\mrm{poly}}(C) = \boplus_{p = -1}^{n-1}
\mcal{T}^p_{\mrm{poly}}(C)  \]
of $C$ relative to $\K$, where $n := \opn{dim} C$. 
It is the exterior algebra over $C$ of the module of derivations
$\mcal{T}(C)$, but with a shift in degrees:
$\mcal{T}^p_{\mrm{poly}}(C) :=
\bwedge^{p+1}_C \mcal{T}(C)$.
The differential is zero, and the Lie bracket is the
Schouten-Nijenhuis bracket, that extends the usual Lie bracket on
$ \mcal{T}(C) = \mcal{T}^0_{\mrm{poly}}(C)$,
and its canonical action $\opn{ad}_C$ on 
$C = \mcal{T}^{-1}_{\mrm{poly}}(C)$
by derivations. 
The DG Lie algebra $\mcal{T}^{}_{\mrm{poly}}(C)$ is of course of
quantum type.

Passing to the DG Lie $R$-algebra
$\m \hot \mcal{T}_{\mrm{poly}}(C)$,
we have an action of the Lie algebra
$\m \hot \mcal{T}^0_{\mrm{poly}}(C)$
on the commutative algebra
$A := R  \hot C$
by $R$-linear derivations, which we denote by $\opn{ad}_A$.
If we choose a filtered 
$\K$-basis $\{ r_j \}_{j \geq 1}$ of $\m$, then 
for
$\gamma  =  \sum_{j \geq 1} r_j \otimes \gamma_j$
and $c \in C$ this action becomes
\[ \opn{ad}_A(\gamma)(c) = 
 \sum_{j \geq 1} r_j \otimes \opn{ad}_C(\gamma_j)(c) \in R \hot C . \]
Here we identify the element $c \in C$ with the tensor
$1_R \otimes c \in A = R \hot C$.
The exponential of $\opn{ad}_A(\gamma)$ is an  automorphism
$\opn{exp}(\opn{ad}_A(\gamma))$ 
of the $R$-module $A = R  \hot C$, as in (\ref{eqn:dglie.118}).

An element $\om \in \mcal{T}^1_{\mrm{poly}}(C)$ determines an antisymmetric
bilinear function $\{ -,- \}_{\om}$ on $C$. The formula 
for $\om = \ga_1 \wedge \ga_2$ is 
\begin{equation} \label{equ:301}
\{ c_1, c_2 \}_{\om} := 
\smfrac{1}{2} \bigl( \al_1(c_1) \al_2(c_2) - \al_1(c_2) \al_2(c_1) \bigr) 
\end{equation}
for $c_1, c_2 \in C$. 
Now take an element
$\om \in \m \hot \mcal{T}^1_{\mrm{poly}}(C)$.
By extending (\ref{equ:301}) $R$-linearly we get an
antisymmetric $R$-bilinear function
$\{ -,- \}_{\om}$ on $R \hot C$. 
If the expansion of $\om$ is 
$\om = \sum_{j \geq 1} r_j \otimes \om_j$,
then
\begin{equation} \label{eqn:extras.15}
\{ c_1, c_2 \}_{\om} :=
\sum_{j \geq 1} r_j \otimes \{ c_1, c_2 \}_{\om_j} 
\in R \hot C .
\end{equation}

\begin{dfn} \label{dfn:dglie.110}
Consider the commutative $R$-algebra 
$A := R \hot C$, with the obvious augmentation
$\psi : \K \otimes_R A \iso C$.
\begin{enumerate}
\item A {\em formal Poisson bracket} on $A$ is an $R$-bilinear
Poisson bracket that vanishes modulo $\m$. 
\item A {\em gauge transformation} of $A$ (as 
$R$-algebra) is an $R$-algebra automorphism that commutes with the
augmentation to $C$.
\end{enumerate}
\end{dfn}

According to Proposition \ref{prop:12} the commutative $R$-algebra 
$A := R \hot C$ is flat and $\m$-adically complete.
Therefore, by endowing it with a formal Poisson bracket $\om$, we
obtain a Poisson $R$-deformation of $C$, and we denote this 
deformation by $A_{\om}$.

The next  result summarizes the role of $\mcal{T}^{}_{\mrm{poly}}(C)$
in controlling formal Poisson brackets. 

\begin{prop} \label{prop:4}
Consider the augmented commutative $R$-algebra 
$A := R \hot C$.
\begin{enumerate}
\item The formula
$\opn{exp}(\gamma) \mapsto \opn{exp}(\opn{ad}_A(\gamma))$
determines a group isomorphism from 
$\opn{exp}\bigl( \m \hot \mcal{T}^0_{\mrm{poly}}(C) \bigr)$
to the group of gauge transformations of $A$ \tup{(}as augmented
$R$-algebra\tup{)}. 
\item The formula
$\om \mapsto \{ -,- \}_{\om}$
determines a bijection from 
$\opn{MC}\bigl( \m \hot \mcal{T}_{\mrm{poly}}(C) \bigr)$
to the set of formal Poisson brackets on $A$. For such $\om$ we
denote by $A_{\om}$ the corresponding
Poisson algebra.
\item Let 
$\om, \om' \in \opn{MC}\bigl( \m \hot
\mcal{T}_{\mrm{poly}}(C) \bigr)$, and
let $\gamma \in \m \hot \mcal{T}^0_{\mrm{poly}}(C)$.
Then 
$\om' = \opn{exp}(\opn{af}(\gamma))(\om)$
if and only if
$\opn{exp}(\opn{ad}_A(\gamma)) : A_{\om} \to A_{\om'}$
is a gauge transformation of Poisson deformations.
\item For 
$\om \in \opn{MC}\bigl( \m \hot
\mcal{T}_{\mrm{poly}}(C) \bigr)$, 
one has equality of groups
\[ \opn{IG}(A_{\om}) = 
\opn{exp}\bigl( \m \hot \mcal{T}^{-1}_{\mrm{poly}}(C) 
\bigr)_{\om} . \]
\end{enumerate}
\end{prop}

\begin{proof}
(1) By definition the operator $\opn{ad}_A(\gamma)$ is a pronilpotent
derivation of the $R$-algebra $A$. According to \cite[Section 2.3]{Hu}
the operator $g := \opn{exp}(\opn{ad}_A(\gamma))$ is an $R$-algebra
automorphism of $A$. Since 
$\opn{ad}_C : \mcal{T}(C) \to \opn{End}(C)$
is an injective Lie algebra homomorphism, it follows that 
\[ \opn{exp}(\opn{ad}_A(-)) : \exp(\m \hot \mcal{T}(C)) \to 
\exp(\m \hot \opn{End}(C)) \]
is an injective group homomorphism. 

Now suppose $g$ is a gauge transformation of $A$ as augmented $R$-algebra. 
Lemma \ref{lem:dglie.111} says we can view $g$ as an element of 
$R \hot \opn{End}(C)$. We
will produce a sequence $\gamma_i \in \m \hot \mcal{T}(C)$
such that 
$g \equiv \opn{exp}(\opn{ad}_A(\gamma_i))$ modulo $\m^{i+1}$.
Then for $\gamma := \opn{lim}_{i \to  \infty} \gamma_i$ 
we will have 
$g = \opn{exp}(\opn{ad}_A(\gamma))$.
Here is the construction. We start with $\gamma_0 := 0$ of course.
Next assume that we have $\gamma_i$. There is a unique element
$\delta_{i+1} \in (\m^{i+1} / \m^{i+2}) \ot \opn{End}(C)$
such that 
$g \circ \opn{exp}(\opn{ad}_A(\gamma_i))^{-1} = \bsym{1} +
\delta_{i+1}$
as automorphisms of the $R_{i+1}$-algebra
$A_{i+1} := R_{i+1} \ot C$. 
The usual calculation shows that $\delta_{i+1}$ is a derivation, i.e.\ 
$\delta_{i+1} \in (\m^{i+1} / \m^{i+2}) \ot \mcal{T}(C)$.
Choose some lifting 
$\til{\delta}_{i+1} \in \m^{i+1}  \hot \mcal{T}(C)$
of $\delta_{i+1}$, and define
$\gamma_{i+1} := \gamma_i + \til{\delta}_{i+1}$.

\medskip \noindent
(2, 3) See \cite[paragraph 4.6.2]{Ko1} or 
\cite[paragraph 3.5.3]{CKTB}.

\medskip \noindent
(4) By definition 
$\opn{IG}(A_{\om}) = \exp(\m A_{\om})$.
\end{proof}

The associative case is much more difficult. 
When dealing with associative deformations we
view $A := R \hot C$ as an $R$-module.
The augmentation $A \to C$ is viewed as a homomorphism of
$R$-modules, and there is a distinguished element
$1_A := 1_R \otimes 1_C \in A$. 

\begin{dfn}
Consider the augmented $R$-module
$A := R \hot C$, with distinguished element $1_A$.
\begin{enumerate}
\item A {\em star product} on $A$ is an $R$-bilinear function
$\star : A \times A \to A$
that makes $A$ into an associative $R$-algebra, with unit 
$1_A$, such that 
$c_1 \star c_2 \equiv c_1 c_2 \ \opn{mod} \m$
for $c_1, c_2 \in C$.
\item A {\em gauge transformation} of $A$ (as $R$-module) is an $R$-module
automorphism that commutes with the augmentation to $C$ and fixes the element
$1_A$.
\end{enumerate}
\end{dfn}

Given a star product $\star$ on $A$, we have an
associative $R$-deformation of $C$.
If we choose a filtered  $\K$-basis
$\{ r_j \}_{j \geq 1}$ of $\m$, then we can express $\star$ as
a power series
\[ c_1 \star c_2 = c_1 c_2 + 
\sum_{j \geq 1} r_j \, \om_j(c_2, c_2) 
\in A , \]
where 
$\om_j \in \opn{Hom}(C \ot C, C)$.

Star products are controlled by a DG Lie algebra too. It is the
{\em shifted Hochschild cochain complex}
\begin{equation}
\mcal{C}_{\mrm{shc}}(C) = \boplus_{p \geq -1}
\mcal{C}_{\mrm{shc}}^p(C) ,
\end{equation}
where
\begin{equation}
\mcal{C}_{\mrm{shc}}^p(C) := 
\opn{Hom}( \underset{p+1}{\underbrace{C \ot \cdots \ot C}}, C)
\end{equation}
for $p \geq 0$, and 
$\mcal{C}_{\mrm{shc}}^{-1}(C) := C$.
The differential is the shift of the Hochschild differential, and the
Lie bracket is the Gerstenhaber bracket. 
(In our earlier paper \cite{Ye2} we used the notation 
$\mcal{C}_{\mrm{dual}}(C)[1]$ for this DG Lie algebra.) 
Inside $\mcal{C}_{\mrm{shc}}(C)$ there is a sub DG Lie algebra 
$\mcal{C}_{\mrm{shc}}^{\mrm{nor}}(C)$, consisting of the {\em
normalized
cochains}. By definition a cochain 
$\phi \in \mcal{C}_{\mrm{shc}}^p(C)$ is normalized if either $p = -1$,
or $p \geq 0$ and 
$\phi(c_1 \otimes \cdots \otimes c_{p+1}) = 0$
whenever $c_i = 1$ for some index $i$. 

Given $\om \in \m \hot 
\mcal{C}_{\mrm{shc}}^{\mrm{nor}, 1}(C)$ we
denote by
$\star_{\om}$ the $R$-bilinear function on the $R$-module
$A := R \hot C$ with formula
\begin{equation} \label{eqn:dglie.25}
c_1 \star_{\om} c_2 := c_1 c_2 + \om(c_1, c_2)
\end{equation}
for $c_1, c_2 \in C$. 
And for $\gamma \in \m \hot 
\mcal{C}_{\mrm{shc}}^{\mrm{nor}, 0}(C)$
we denote by $\opn{ad}_A$ the $R$-linear function on $A$
such that
$\opn{ad}_A(c) := [\gamma, c] = \gamma(c)$
for $c \in C$.

According to Proposition \ref{prop:12}, any associative
$R$-deformation $A$ of $C$ is isomorphic,
as augmented $R$-module with distinguished element $1_A$, to $R \hot C$
with its distinguished element $1_R \ot 1_C$.
Like Proposition \ref{prop:4}, we have:

\begin{prop} \label{prop:5}
Consider the augmented $R$-module $A := R \hot C$ with
distinguished element $1_A := 1_R \ot 1_C$.
\begin{enumerate}
\item The formula
$\opn{exp}(\gamma) \mapsto \opn{exp}(\opn{ad}_A(\gamma))$
determines a group isomorphism from 
$\opn{exp}\bigl( \m \hot 
\mcal{C}_{\mrm{shc}}^{\mrm{nor}, 0}(C) \bigr)$
to the group of gauge transformations of the augmented $R$-module $A$
that preserve $1_A$. 
\item The formula
$\om \mapsto \star_{\om}$
determines a bijection from 
$\opn{MC}\bigl( \m \hot 
\mcal{C}_{\mrm{shc}}^{\mrm{nor}, 1}(C) \bigr)$
to the set of star products on $A$.
For such $\om$ we denote by $A_{\om}$ the resulting associative
$R$-algebra.
\item Let 
$\om, \om' \in \opn{MC}\bigl( \m \hot
\mcal{C}_{\mrm{shc}}^{\mrm{nor}}(C) \bigr)$, and 
let 
$\gamma \in \m \hot 
\mcal{C}_{\mrm{shc}}^{\mrm{nor}, 1}(C)$.
Then 
$\om' = \opn{exp}(\opn{af}(\gamma))(\om)$
if and only if
$\opn{exp}(\opn{ad}_A( \gamma)) : A_{\om} \to A_{\om'}$
is a gauge transformation of associative $R$-deformations of $C$.
\item For 
$\om \in \opn{MC}\bigl( \m \hot
\mcal{C}_{\mrm{shc}}^{\mrm{nor}, 1}(C) \bigr)$, 
there is a canonical isomorphism of groups
\[ \opn{IG}(A_{\om}) \cong
\opn{exp}\bigl( \m \hot 
\mcal{C}_{\mrm{shc}}^{\mrm{nor}, -1}(C) 
\bigr)_{\om} . \]
\end{enumerate}
\end{prop}

\begin{proof}
(1) Combine Lemma \ref{lem:dglie.111} with the observation that the
automorphism \lb $\opn{exp}(\opn{ad}_A(\gamma))$, for 
$\ga \in \m \hot \mcal{C}_{\mrm{shc}}^{0}(C) = \m \hot \opn{End}(C)$,
fixes the element $1_A$ if and only if $\gamma$ is normalized.

\medskip \noindent
(2, 3) See \cite[paragraphs 3.4.2 and 4.6.2]{Ko1} or 
\cite[Section 3.3]{CKTB}. Cf.\ also \cite[Propositions 3.20 and
3.21]{Ye1}.

\medskip \noindent
(4) 
By definition 
$\opn{IG}(A_{\om}) = \exp(\m A_{\om})$.
\end{proof}

%\cleardoublepage
\section{Polydifferential Operators} \label{sec:ploy-diff}

We continue with Setup \ref{setup:3}. In this section we prove that
associative deformations are actually controlled by a sub DG
Lie algebra $\mcal{D}^{\mrm{nor}}_{\mrm{poly}}(C)$ of
$\mcal{C}_{\mrm{shc}}^{\mrm{nor}}(C)$, 
which has better behavior. 

Take a Hochschild cochain 
$\phi \in \mcal{C}_{\mrm{shc}}^{p}(C)$
for some $p \geq 0$. The function $\phi$ is called a {\em 
polydifferential operator} if there is a number $m \in \N$, such that for every 
$i$ and every $c_1, \ldots, c_{p+1} \in C$, the function
\[ c \mapsto \phi(c_1, \ldots, c_{i-1}, c, c_{i+1}, \ldots, c_{p+1}) \]
is a differential operator $C \to C$ of order $\leq m$. We denote by
$\mcal{D}^{p}_{\mrm{poly}}(C)$ the set of these polydifferential
operators. And we let
$\mcal{D}^{-1}_{\mrm{poly}}(C) := C$.
Then $\mcal{D}^{}_{\mrm{poly}}(C)$ is a sub DG Lie algebra of the shifted
Hochschild cochain complex 
$\mcal{C}_{\mrm{shc}}(C)$. We define a yet smaller DG Lie algebra
\[ \mcal{D}^{\mrm{nor}}_{\mrm{poly}}(C) :=
\mcal{D}^{}_{\mrm{poly}}(C) \cap 
\mcal{C}_{\mrm{shc}}^{\mrm{nor}}(C) , \]
whose elements are the {\em normalized polydifferential operators}.

\begin{dfn} \label{dfn:16}
Consider the augmented $R$-module
$A := R \hot C$,
with distinguished element $1_A$. Recall the bijections of Proposition
\ref{prop:5}(1-2).
\begin{enumerate}
\item A {\em formal polydifferential operator} on $A$ is an element 
$\phi \in \m \hot \mcal{D}^{p}_{\mrm{poly}}(C)$ for some $p \geq 0$.

\item A gauge transformation $g : A \to A$ is called a 
{\em differential gauge transformation} if 
$\gamma := \opn{log}(g)$ is a formal differential operator, i.e.\ 
$\ga \in \m \hot \mcal{D}^{\mrm{nor}, 0}_{\mrm{poly}}(C)$.

\item A star product $\star$ on $A$ is called a
{\em differential star product} if the corresponding MC element
$\om$ is a formal bidifferential operator, i.e.\
$\om \in \m \hot \mcal{D}^{\mrm{nor}, 1}_{\mrm{poly}}(C)$.
\end{enumerate}
\end{dfn}

\begin{thm} \label{thm:4}
Assume $R$ and $C$ are as in Setup \tup{\ref{setup:3}}. Then
any star product on the $R$-module $A := R \hot C$
is gauge equivalent to a differential star product. 
Namely, given a star product $\star$ on $A$, there exists a 
gauge transformation $g : A \to A$, and a 
differential star product $\star'$, such that
\begin{equation} \label{eqn:10}
g(a_1 \star a_2) = g(a_1) \star' g(a_2)  
\end{equation}
for any $a_1, a_2 \in A$.
\end{thm}

\begin{proof}
This is a mild generalization of \cite[Proposition 8.1]{Ye1},
which refers to $R = \K[[\hbar]]$.
According to \cite[Corollary 4.12]{Ye2}, the inclusion
$\mcal{D}^{\mrm{nor}}_{\mrm{poly}}(C) \to 
\mcal{C}_{\mrm{shc}}^{\mrm{nor}}(C)$
is a quasi-isomorphism. Therefore, by \cite[Theorem 4.2]{Ye5}, we get a
bijection
\[ \ol{\mrm{MC}} \bigl( \m \hot 
\mcal{D}^{\mrm{nor}}_{\mrm{poly}}(C) \bigr) \to 
\ol{\mrm{MC}} \bigl( \m \hot 
\mcal{C}_{\mrm{shc}}^{\mrm{nor}}(C) \bigr) . \]
Let 
$\om \in \mrm{MC} \bigl( \m \hot 
\mcal{C}_{\mrm{shc}}^{\mrm{nor}}(C) \bigr)$
be the element representing $\star$; see Proposition  \ref{prop:5}(2).
Next let 
$\om' \in \mrm{MC} \bigl( \m \hot 
\mcal{D}^{\mrm{nor}}_{\mrm{poly}}(C) \bigr)$
be an element that's gauge equivalent to $\om$. By 
Proposition \ref{prop:5}(3) there is a gauge transformation
$g := \opn{exp}(\opn{ad}_A( \gamma))$ which satisfies equation
(\ref{eqn:10}). 
\end{proof}

\begin{rem}
It should be noted that the proof of \cite[Corollary 4.12]{Ye2}
relies on the fact that $C$ is a smooth $\K$-algebra and 
$\opn{char} \K = 0$. The result is most likely false otherwise.
\end{rem}

We learned the next result from P. Etingof. It is very similar
to \cite[Proposition 2.2.3]{KS}. 

\begin{thm}  \label{thm:5}
Assume $R$ and $C$ are as in Setup \tup{\ref{setup:3}}. 
Let $\star$ and $\star'$ be two star products on the augmented $R$-module 
$A := C \hatotimes{} R$, and let $g$ be a gauge transformation of $A$
satisfying \tup{(\ref{eqn:10})}. Assume that $\star$ is a differential star
product. The following conditions are equivalent\tup{:}
\begin{enumerate}
\rmitem{i} The star product $\star'$ is also differential.

\rmitem{ii} The gauge transformation  $g$ is differential.
\end{enumerate}
\end{thm}

\begin{proof}
The implication (ii) $\Rightarrow$ (i) is easy: the subgroup 
$\exp \bigl( \m \hot \mcal{D}^{\mrm{nor}, 0}_{\mrm{poly}}(C) \bigr)$
of 
$\exp\bigl( \m \hot \mcal{C}_{\mrm{shc}}^{\mrm{nor}}(C) \bigr)$
acts on the subset 
$\mrm{MC} \bigl( \m \hot \mcal{D}^{\mrm{nor}}_{\mrm{poly}}(C) \bigr)$
of 
$\mrm{MC} \bigl( \m \hot \mcal{C}_{\mrm{shc}}^{\mrm{nor}}(C) \bigr)$.

We now consider (i) $\Rightarrow$ (ii); so $\star'$ is differential.
Let us choose a filtered $\K$-basis
$\{ r_i \}_{i \geq 0}$ of $R$, such that $r_0 = 1$, and 
$\opn{ord}_{\m}(r_i) \leq \opn{ord}_{\m}(r_{i+1})$. 
Denote by 
$\{ \mu_{i,j; k} \}_{i,j,k \geq 0}$ 
the multiplication constants of the basis
$\{ r_i \}_{i \geq 0}$, i.e.\ the collection of
elements of $\K$ such that
$r_i r_j = \sum_k \mu_{i,j; k}\, r_k$.
Note that 
$\mu_{0,j; j} = \mu_{i,0; i} = 1$, and
$\mu_{i,j; k} = 0$ if $i+j > k$.

The gauge transformation $g$ has an expansion
$g = \sum_{i \geq 0} r_i \otimes \gamma_i$,
with $\gamma_0 = \bsym{1}_C$ and
$\gamma_i \in  \mcal{C}_{\mrm{shc}}^{\mrm{nor}, 0}(C) \subset
\opn{End}(C)$ for $i \geq 1$.
We will begin by showing that 
$\gamma_i$ are differential operators. 
This calculation is by induction on $i$, and it is almost identical to
the proof of \cite[Proposition 4.3]{KS}. 

Let us denote by 
$\om_i, \om'_i \in \mcal{D}^{1}_{\mrm{poly}}(C)$
the bidifferential operators such that 
\[ c \star d = \sum_{i \geq 0} r_i \otimes \om_i(c,d) \]
and
\[ c \star' d = \sum_{i \geq 0} r_i \otimes \om'_i(c,d) \]
for all $c, d \in C$. Thus
$\om_0(c,d) = \om'_0(c,d) =  c d$,
and 
$\om_i, \om'_i \in \mcal{D}^{\mrm{nor}, 1}_{\mrm{poly}}(C)$
for $i \geq 1$.
By expanding the two sides of (\ref{eqn:10}) we get
\[ g(c \star d) = \sum_{i \geq 0} r_i \otimes
\Bigl( \sum_{j + k \leq i} \mu_{j,k; i}\, \gamma_k\bigl(
\om_j(c,d) \bigr) \Bigr) \]
and
\[ g(c) \star' g(d) = 
\sum_{i \geq 0} r_i \otimes
\Bigl( \sum_{m + l \leq i} \ \sum_{j + k  \leq m} 
\mu_{j,k; m}\, \mu_{m,l; i} \,
\om'_l \bigl( \gamma_j(c), \gamma_k(d) \bigr) \Bigr) . \]

Now we compare the coefficients of $r_i$, for $i \geq 1$, in these
last two equations:
\begin{equation} \label{eqn:14}
\sum_{j + k \leq i} \mu_{j,k; i}\, \gamma_k \bigl(
\om_j(c, d) \bigr) =
\sum_{m + l \leq i} \ \sum_{j + k  \leq m} 
\mu_{j,k; m}\, \mu_{m, l; i}\, 
\om'_l \bigl( \gamma_j(c), \gamma_k(d) \bigr) .
\end{equation}
We take the summand with $k = i$ (and $j = 0$) in the left side
of (\ref{eqn:14}), and subtract from it the summand with 
$j = m = i$ (and $k = l = 0$)
in the right side of that equation. This yields
\[ \mu_{0,i; i}\, \gamma_i \bigl(\om_0(c, d) \bigr) -
\mu_{i,0; i}\, \mu_{i,0; i}\, 
\om'_0 \bigl( \gamma_i(c), \gamma_0(d) \bigr) =
\phi_i(c, d) , \]
where $\phi_i(c, d)$ involves the bidifferential operators 
$\om_k, \om'_k$, and the operators 
$\gamma_j$ for $j < i$, which are differential by the induction
hypothesis. We see that $\phi_i(c, d)$ is itself a bidifferential
operator, say of order $\leq m_i$ in each argument. 
And since $\mu_{0,i; i} = 1$ etc., we have
$\gamma_i(c d) - \gamma_i(c) d = \phi_i(c, d)$.

Now, letting $c$ vary, the last equation reads
$[\gamma_i, d] = \phi_i(-, d) \in \opn{End}(C)$.
Hence $[\gamma_i, d]$ is a differential operator,
also of order $\leq m_i$. This is true for every $d \in C$.
By Grothendieck's characterization of differential
operators, it follows that $\gamma_{i}$ is a differential operator (of
order $\leq m_i + 1$). 

Finally let us consider $\opn{log}(g)$. 
We know that 
$r_i \otimes \gamma_i \in \m \hot 
\mcal{D}^{\mrm{nor}, 0}_{\mrm{poly}}(C)$
for $i \geq 1$. And 
$\m \hot \mcal{D}^{\mrm{nor}, 0}_{\mrm{poly}}(C)$
is a closed (nonunital) subalgebra of the ring $R \hot \opn{End}(C)$.
By plugging
$x := \sum_{i \geq 1} r_i \otimes \gamma_i$ into the usual power
series 
$\opn{log}(1+x) = x - \smfrac{1}{2} x^2 +  \cdots $
we conclude that 
$\opn{log}(g) \in 
\m \hot \mcal{D}^{\mrm{nor}, 0}_{\mrm{poly}}(C)$.
\end{proof}

%\cleardoublepage
\section{Deformations of Affine Varieties} \label{sec:defs-vars}

In this section we assume the following setup (a special case of Setup
\ref{setup:230}):

\begin{setup} \label{setup:defs-sh.110}
$\K$ is a field of characteristic $0$; $(R, \m)$ is a parameter 
algebra over $\K$; and $X$ is a smooth algebraic variety
over $\K$, with structure sheaf $\mcal{O}_X$.
\end{setup}

The sheaf $R \hot \mcal{O}_X$ is viewed either as a sheaf of commutative
$R$-algebras, or as a sheaf of $R$-modules with distinguished global section 
$1_R \ot 1_{\mcal{O}_X}$, depending on whether we are dealing with the Poisson
case or the associative case. In both cases there is an augmentation
$R \hot \mcal{O}_X \to \mcal{O}_X$. Thanks to Corollary \ref{cor:240}
we know that $R \hot \mcal{O}_X$ is flat over $R$ and $\m$-adically complete.
Also for every affine open set $U \subset X$ the canonical homomorphism 
\begin{equation} \label{eqn:defs-sh.110}
R \hot \Gamma(U, \mcal{O}_X) \to \Gamma(U, R \hot \mcal{O}_X)
\end{equation}
is bijective. 

Let $U \subset X$ be an affine open set and $C := \Gamma(U, \mcal{O}_X)$.
Recall that an element of $\m \hot \mcal{D}^{\mrm{nor}, 0}_{\mrm{poly}}(C)$
is called a {\em formal differential operator} of $R \hot C$, and an 
element of $\m \hot \mcal{T}^{0}_{\mrm{poly}}(C)$
is called a {\em formal derivation}. 
A {\em differential gauge transformation} of $R \hot C$ is an
$R$-linear automorphism of the form 
$\exp(\ga)$, for some formal differential operator $\ga$. 
Note that $\mcal{T}^{0}_{\mrm{poly}}(C)$ can be viewed as a submodule of 
$\mcal{D}^{\mrm{nor}, 0}_{\mrm{poly}}(C)$. 
According to Proposition \ref{prop:4}(1), the differential gauge transformation
$\exp(\ga)$ preserves the commutative multiplication of $R \hot C$ if and only
if $\ga$ is a formal derivation.

Differential star products and formal Poisson brackets on $R \hot C$ were
introduced in Definitions \ref{dfn:16} and \ref{dfn:dglie.110} respectively.
By abuse of notation, given 
$\om \in \opn{MC}(\m \hot \mcal{D}^{\mrm{nor}}_{\mrm{poly}}(C))$,
we call $\om$ a differential star product, even though the actual star product
is $\star_{\om}$, as in (\ref{eqn:dglie.25}).

We now go to sheaves. A differential gauge transformation $g$ of 
$R \hot \mcal{O}_X$ is an $R$-linear sheaf automorphism, such that for every
affine open set  $U \subset X$, the automorphism of 
$R \hot \Gamma(U, \mcal{O}_X)$ induced by $g$ through the canonical isomorphism 
(\ref{eqn:defs-sh.110}) is a differential gauge transformation.
Similarly, a differential star product (resp.\ formal Poisson bracket) 
on $R \hot \mcal{O}_X$ is an $R$-bilinear pairing $\om$, such that for every
affine open set $U \subset X$, the pairing on 
$R \hot \Gamma(U, \mcal{O}_X)$ induced by $\om$ through the canonical
isomorphism (\ref{eqn:defs-sh.110}) is a differential star product (resp.\
formal Poisson bracket).

\begin{lem} \label{lem:235}
Let $U \subset X$ be an affine open set and $C := \Gamma(U, \mcal{O}_X)$.
\begin{enumerate}
\item Let $g$ be a differential gauge transformation of the augmented $R$-module
$R \hot C$. Then $g$ extends uniquely to a differential gauge transformation 
$\til{g}$ of the sheaf of $R$-modules $R \hot \mcal{O}_U$.

\item Let $\om$ be a differential star product \tup{(}resp.\
formal Poisson bracket\tup{)} on $R \hot C$. Then $\om$ extends uniquely to a 
differential star product \tup{(}resp.\ formal Poisson bracket\tup{)}
$\til{\om}$ on $R \hot \mcal{O}_U$. We denote by 
$(R \hot \mcal{O}_U)_{\til{\om}}$ the resulting $R$-deformation of $\mcal{O}_U$.

\item Suppose $\om$ and $\om'$ are differential star products \tup{(}resp.\
formal Poisson brackets\tup{)} on $R \hot C$, and $g$ is a differential gauge
transformation of the augmented $R$-module $R \hot C$, which is also 
a gauge transformation 
$g : (R \hot C)_{\om} \to (R \hot C)_{\om'}$
of $R$-deformations of $C$. Then 
\[ \til{g} : (R \hot \mcal{O}_U)_{\til{\om}} \to 
(R \hot \mcal{O}_U)_{\til{\om}'} \]
is a gauge transformation of $R$-deformations of $\mcal{O}_U$.
\end{enumerate}
\end{lem}

\begin{proof}
(1)-(2). Both assertions follow from the fact that 
$\mcal{D}_{\mrm{poly}, X}$ is a sheaf of DG Lie algebras on $X$, and each 
$\mcal{D}^p_{\mrm{poly}, X}$ is a quasi-coherent sheaf. To be more precise, 
consider a formal polydifferential operator $\om$ on $R \hot C$
(see Definition \ref{dfn:16}(1); for item (1) we take $\om := \log(g)$).
Let us denote by $\{ U_k \}_{k \in K}$ the collection of affine open sets of
$U$. Take an index $k \in K$, and let $C_k := \Gamma(U_k, \mcal{O}_X)$.
Then $C \to C_k$ is an \'etale $\K$-algebra homomorphism, so $\om$
extends uniquely to a polydifferential operator $\om_k$ on 
$R \hot C_k$; cf.\ \cite[Proposition 2.7]{Ye2}. 
Uniqueness implies that $\om_k$ has the same algebraic properties (Lie bracket,
star product) as $\om$.
Since the collection of open sets $\{ U_k \}_{k \in K}$ is a
basis of the topology of $U$, the collection of operators 
$\{ \om_k \}_{k \in K}$ determines an 
operator $\til{\om}$ on the sheaf $R \hot \mcal{O}_U$, whose restriction to 
$\Gamma(U_k, R \hot \mcal{O}_U)$ is $\om_k$. So $\til{\om}$ is a
formal polydifferential operator on $R \hot \mcal{O}_U$. 

\medskip \noindent
(3) Take any affine open set $U_k \subset U$, and let $C_k$ be as above.
We have to prove that for any 
$c_1, c_2 \in C_k$ there is equality 
$g_k( \om_k(c_1, c_2) ) = \om'_k( g_k(c_1), g_k(c_2) )$
in $R \hot C_k$. But both sides are formal bidifferential operators applied to
the pair $(c_1, c_2)$; so this is also a consequence of the uniqueness of
extension of formal polydifferential operators mentioned above.
\end{proof}

\begin{lem} \label{lem:defs-sh.110}
Let $U \subset X$ be an affine open set and $C := \Gamma(U, \mcal{O}_X)$.
Suppose $A$ is a Poisson $R$-deformation of $C$. Then there is an 
isomorphism of augmented commutative $R$-algebras 
$R \hot C \cong A$.
\end{lem}

\begin{proof}
We write $R_i := R / \m^{i+1}$ for $i \geq 0$.
Since $C$ is formally smooth over $\K$, 
we can find a compatible family of $\K$-algebra liftings
$C \to R_i \otimes_R A$ of the augmentation. 
Due to flatness the induced $R_i$-algebra 
homomorphisms 
$R_i \ot C \to R_i \otimes_R A$
are bijective. And because $A$ is complete we get an isomorphism of
augmented $R$-algebras
$R \hot C \iso A$ in the limit. 
\end{proof}

\begin{lem} \label{lem:240}
Let $R$ be a parameter algebra over $\K$. Take an affine open set
$U \subset X$, and let $C := \Gamma(U, \mcal{O}_X)$.

\begin{enumerate}
\item Let $\mcal{A}$ be a Poisson \tup{(}resp.\ associative\tup{)}
$R$-deformation of $\mcal{O}_X$. Then there is a formal Poisson bracket
\tup{(}resp.\ differential star product\tup{)} $\om$ on $R \hot C$, and a gauge
transformation 
\[ \til{g} : (R \hot \mcal{O}_U)_{\til{\om}} \to \mcal{A}|_U \]
between Poisson \tup{(}resp.\ associative\tup{)} $R$-deformations of
$\mcal{O}_U$.

\item Let $\om$ and $\om'$ be formal Poisson brackets \tup{(}resp.\ differential
star products\tup{)} on $R \hot C$, and let 
\[ \til{g} : (R \hot \mcal{O}_U)_{\til{\om}} \to 
(R \hot \mcal{O}_U)_{\til{\om}'} \]
be a gauge transformation between  Poisson \tup{(}resp.\ associative\tup{)} \lb
$R$-deformations of $\mcal{O}_U$.
Then $\til{g}$ is the extension of the gauge transformation
$g = \exp(\ga)$, for a unique formal derivation \tup{(}resp.\ 
formal differential operator\tup{)} $\ga$ of $R \hot C$.
\end{enumerate}
\end{lem}

\begin{proof}
(1) Let $A := \Gamma(U, \mcal{A})$, which by Proposition \ref{prop:233} is a
Poisson (resp.\ associative) $R$-deformation of $C$. In the Poisson case there
is an isomorphism $g : R \hot C \to A$ of augmented $R$-algebras, by Lemma
\ref{lem:defs-sh.110}. According to Proposition \ref{prop:4}(2) there is a
formal Poisson bracket $\om$, such that
$g : (R \hot C)_{\om} \to A$ is a gauge transformation of Poisson deformations
of $C$. 

In the associative case we know from Proposition \ref{prop:12} that there is an 
isomorphism $g' : R \hot C \to A$ of augmented $R$-modules, sending  
$1_R \ot 1_C \mapsto 1_A$. By Theorem \ref{thm:4} we can change $g'$
to another isomorphism $R$-modules $g : R \hot C \to A$, such that 
$g : (R \hot C)_{\om} \to A$ is a gauge transformation of associative
$R$-deformations of $C$, for some differential star product $\om$. 

In both cases we now use Lemma \ref{lem:235}(2) to deduce
that the deformation $(R \hot C)_{\om}$ of $C$ extends to a deformation 
$(R \hot \mcal{O}_U)_{\om}$ of $\mcal{O}_U$. 
There is a gauge transformation
$g : \Gamma \bigl( U, (R \hot \mcal{O}_U)_{\om} \bigr) \to 
\Gamma(U, \mcal{A})$
of deformations of $C$. According to Theorem \ref{thm:defs-sh.112}(2), this
extends to a gauge transformation 
$g : (R \hot \mcal{O}_U)_{\om} \to \mcal{A}|_U$
of deformations of $\mcal{O}_U$. 

\medskip \noindent
(2) The delicate issue here is that a priori we don't know that $\til{g}$ is a
differential gauge transformation. 

Applying $\Gamma(U, -)$ to $\til{g}$ we get a gauge transformation
$g : (R \hot C)_{\om} \to \lb (R \hot C)_{\om'}$
between $R$-deformations of $C$. In the Poisson case we know from 
Proposition \ref{prop:4}(1) that 
$g = \exp(\ga)$, for a unique formal derivation $\ga$. 
In the associative case, Theorem \ref{thm:5} says that 
$g = \exp(\ga)$ for a unique formal differential operator $\ga$. 
Next, in both cases, using Lemma \ref{lem:235}(1,3), we see that
$g$ extends uniquely to a differential gauge transformation 
$\til{g}' : (R \hot \mcal{O}_U)_{\om} \to (R \hot \mcal{O}_U)_{\om'}$
between $R$-deformations of $\mcal{O}_U$. 

To finish the proof, the uniqueness in  Theorem \ref{thm:defs-sh.112}(2)
tells us that $\til{g}' = \til{g}$.
\end{proof}

The notion of equivalence of crossed groupoids was defined in Definition
\ref{dfn:221}). 
The dependence of $\cat{AssDef}(R, \mcal{O}_U)$ and
$\cat{PoisDef}(R, \mcal{O}_U)$ on $U$ and $R$ was explained in Proposition
\ref{prop:239}.
Here is the main result of the paper. 

\begin{thm} \label{thm:202}
Let $\K$ be a field of characteristic $0$, $X$ a smooth algebraic variety
over $\K$, $R$ a parameter algebra over $\K$, 
and $U$ an affine open set in $X$.
There are equivalences of crossed groupoids
\[  \opn{geo} : \
\opn{Del} \bigl( \Gamma(U, \mcal{D}_{\mrm{poly}, X}^{\mrm{nor}}) , R \big)
\to \cat{AssDef}(R, \mcal{O}_U) \]
and
\[  \opn{geo} : \
\opn{Del} \bigl( \Gamma(U, \mcal{T}_{\mrm{poly}, X}) , R \big)
\to \cat{PoisDef}(R, \mcal{O}_U) \]
which we call {\em geometrization}. The equivalences $\opn{geo}$  commute with 
homomorphisms $R \to R'$ of parameter algebras, and with inclusions of affine
open sets $U' \to U$.
\end{thm}

\begin{proof}
Let us write 
$\g(U)$ for either 
$\Gamma(U, \mcal{D}_{\mrm{poly}, X}^{\mrm{nor}})$
or 
$\Gamma(U, \mcal{T}_{\mrm{poly}, X})$.
Likewise we write 
$\cat{P}(R, U)$ for either $\cat{AssDef}(R, \mcal{O}_U)$ or
$\cat{PoisDef}(R, \mcal{O}_U)$.

Take an object 
\[ \om \in \opn{MC}(\m \hot \g(U)) = 
\opn{Ob} \bigl( \opn{Del} ( \g(U) , R ) \bigr) . \]
According to Lemma \ref{lem:235}(2) there is an $R$-deformation 
$\opn{geo}(\om) := (R \hot \mcal{O}_U)_{\til{\om}}$ of $\mcal{O}_U$.
By Lemma \ref{lem:235}(1, 3) a gauge transformation $g : \om \to \om'$ in 
$\opn{Del}_1(\g(U), R)$ induces to a unique gauge transformation of
deformations
$\opn{geo}_1(g) = \til{g} : \opn{geo}(\om) \to \opn{geo}(\om')$ in 
$\cat{P}(R, U)$.
Thus we get a morphism of groupoids 
\[ \opn{geo}_1 :  \opn{Del}_1(\g(U), R) \to \cat{P}_1(R, U) . \]
Lemma \ref{lem:240}(1) says that $\opn{geo}_1$ is essentially surjective on
objects. Lemma \ref{lem:240}(2) tells us that $\opn{geo}_1$ is bijective on
gauge transformations ($1$-morphisms). So $\opn{geo}_1$ is an equivalence.

Let 
$\om \in \opn{Ob} \bigl( \opn{Del} ( \g(U) , R ) \bigr)$.
By Propositions \ref{prop:4}(4) and \ref{prop:5}(4) we have 
a group isomorphism 
\[ \opn{Del}_2 (\g(U), R)(\om) =
\exp(\m \hot \g^{-1}(U))_{\om} \cong \opn{IG}((R \hot C)_{\om}) , \]
where $C := \Gamma(U, \OX)$. And by Corollary \ref{cor:241} there is a group
isomorphism 
\[ \opn{IG}((R \hot C)_{\om}) \cong 
\Gamma \bigl( U, \opn{IG}((R \hot \mcal{O}_U)_{\til{\om}}) \bigr) 
= \cat{P}_2(R, U)(\opn{geo}(\om)) . \]
In this way we get a fully faithful morphism of groupoids
\[  \opn{geo}_2 :  \opn{Del}_2(\g(U), R) \to \cat{P}_2(R, U) .  \]   
         
The fact that the pair of morphisms $\opn{geo} := (\opn{geo}_1, \opn{geo}_2)$
respects the twistings and the feedbacks is immediate
from the definitions (cf.\ Proposition \ref{prop:238} and Definition
\ref{dfn:Lie-desc.101}). 
So $\opn{geo}$ is an equivalence of crossed groupoids.

Finally it is clear from the construction that $\opn{geo}$ is functorial in $U$
and $R$. 
\end{proof}

\begin{rem} \label{rem:defs-sh.110}
In \cite[Definitions 1.4 and 1.8]{Ye1} we introduced the notion of
differential structure on an associative $R$-deformation $\mcal{A}$
of $\mcal{O}_X$. We said there that one must stipulate the existence
of such a differential structure, and uniqueness was not clear. 
Now, having Theorem \ref{thm:202} at our disposal, we know that any
associative $R$-deformation $\mcal{A}$ of $\mcal{O}_X$ admits a differential
structure. Moreover, any two such differential structures are equivalent.
\end{rem}

Here is a similar theorem (but much easier to prove). 

\begin{thm} \label{thm:235}
Let $\K$ be a field of characteristic $0$, $X$ a smooth algebraic variety
over $\K$, $R$ a parameter algebra over $\K$, 
and $U$ an affine open set in $X$.
Write $C := \Gamma(U, \OX)$. 
There are equivalences of crossed groupoids
\[  \Gamma(U, -) : \
\cat{AssDef}(R, \mcal{O}_U) \to \cat{AssDef}(R, C) \]
and
\[  \Gamma(U, -) : \
\cat{PoisDef}(R, \mcal{O}_U) \to \cat{PoisDef}(R, C) , \]
that  commute with homomorphisms $R \to R'$ of parameter algebras.
\end{thm}

\begin{proof}
It is clear that $\Gamma(U, -)$ is a morphism of crossed groupoids. 
By Corollary \ref{cor:243} we know that there is an equivalence of groupoids
on the $1$-truncations (i.e.\ forgetting $2$-morphisms). 
Corollary \ref{cor:241} implies that the group homomorphism  
$ \opn{IG}(\Gamma(U,\mcal{A})) \to \Gamma(U, \opn{IG}(\mcal{A}))$
is bijective.
\end{proof}

\begin{rem}
The associative case of Theorem \ref{thm:235} is valid in arbitrary
characteristic, since $\opn{IG}(\mcal{A})$ can be described without
exponential maps (See Proposition \ref{prop:234}).
\end{rem}

%\cleardoublepage

\end{document}